\documentclass{article}
\usepackage[english]{babel}
\usepackage[letterpaper,top=2cm,bottom=2cm,left=3cm,right=3cm,marginparwidth=1.75cm]{geometry}
\usepackage{graphicx} 
\usepackage{amsmath,amssymb,mathtools,amsthm}

\usepackage[inline]{enumitem}

\usepackage[colorlinks=true, allcolors=blue]{hyperref}
\usepackage{cleveref}
\usepackage{xparse}
\usepackage[style=numeric,giveninits=true,maxbibnames=99]{biblatex}
\usepackage{csquotes}
\usepackage{comment}
\renewbibmacro{in:}{}
\newcommand{\R}{\mathbb R}                              
\newcommand{\C}{\mathbb C}                              
\newcommand{\N}{\mathbb N}                              
\newcommand{\bmat}[1]{\begin{bmatrix}#1\end{bmatrix}}   
\newcommand{\ct}[1]{{#1}^{\mathrm{H}}}                  
\newcommand{\ict}[1]{{#1}^{-\mathrm{H}}}                
\newcommand{\cc}[1]{\overline{#1}}                      
\newcommand{\dd}[2]{\frac{\mathrm{d}#1}{\mathrm{d}#2}}  
\newcommand{\wt}{\widetilde}                            
\newcommand{\wh}{\widehat}                              
\newcommand{\timeInt}{\mathbb T}                        
\newcommand{\td}{\,\mathrm{d}}                          
\newcommand{\avSt}{V_a}                                 
\newcommand{\loc}{\mathrm{loc}}                         
\newcommand{\imagUnit}{\mathrm{i}}                      
\newcommand{\partInt}[1]{                               
    \mathfrak{P}\ifstrempty{#1}{}{_{#1}}}             
\newcommand{\partit}{\mathfrak p}                       
\newcommand{\quadSt}[1]{V_{#1}}
\newcommand{\Cantor}{\mathfrak c}                       
\newcommand{\CantorReverse}{\breve{\Cantor}}            
\newcommand{\restr}[2]{#1|_{#2}}                        

\NewDocumentCommand{\avStArg}{sO{t_0}O{t_1}O{x_0}O{x}O{u}O{y}}{\IfBooleanTF{#1}{#6\,:\,#5(#2)=#4}{\substack{(#3,#6)\,:\,#3\geq#2, \\ #5(#2)=#4}}}
\NewDocumentCommand{\totVar}{oo}{
    \operatorname{Var}\IfNoValueTF{#1}{}{_{#1,#2}}}                      
\NewDocumentCommand{\HerMat}{o}{                        
    \mathbf{S}\IfValueTF{#1}{^{#1}}{}(\C)}
\NewDocumentCommand{\posSD}{o}{                         
    \mathbf{S}\IfValueTF{#1}{^{#1}}{}_+(\C)}
\NewDocumentCommand{\posDef}{o}{                        
    \mathbf{S}\IfValueTF{#1}{^{#1}}{}_{++}(\C)}
\NewDocumentCommand{\GL}{o}{                            
    \operatorname{GL}\IfValueTF{#1}{_{#1}}{}(\C)}

\DeclareMathOperator{\realPart}{Re}                     
\let\ker\relax\DeclareMathOperator{\ker}{Ker}           
\DeclareMathOperator{\BV}{BV}                           
\DeclareMathOperator{\AUC}{AUC}                         
\DeclareMathOperator{\rank}{rank}                       

\DeclarePairedDelimiter{\set}{\{}{\}}
\DeclarePairedDelimiter{\pset}{(}{)}

\DeclarePairedDelimiter{\abs}{\lvert}{\rvert}
\DeclarePairedDelimiter{\aset}{\langle}{\rangle}
\DeclarePairedDelimiter{\norm}{\lVert}{\rVert}

\theoremstyle{definition}
\newtheorem{theorem}{Theorem}[section]
\newtheorem{lemma}[theorem]{Lemma}
\newtheorem{corollary}[theorem]{Corollary}
\newtheorem{definition}[theorem]{Definition}
\newtheorem{remark}[theorem]{Remark}
\newtheorem{example}[theorem]{Example}

\title{Dissipative energy functionals of passive linear time-varying systems}
\author{Riccardo Morandin\footnotemark[1] \and Dorothea Hinsen\footnotemark[2]}
\date{}

\begin{document}

\maketitle
\begin{abstract}
  The concept of dissipativity plays a crucial role in the analysis of control systems.
  Dissipative energy functionals, also known as Hamiltonians, storage functions, or Lyapunov functions, depending on the setting, are extremely valuable to analyze and control the behavior of dynamical systems, but in general circumstances they are very difficult to compute, and not fully understood.
  In this paper we consider passive linear time-varying (LTV) systems, under very mild regularity assumptions, and their associated storage functions, as a necessary step to analyze general nonlinear systems. We demonstrate that every passive LTV system must have at least one time-varying positive semidefinite quadratic storage function, greatly reducing our search scope.
  Now focusing on quadratic storage functions, we analyze in detail their necessary regularity, which is lesser than continuous.
  Moreover, we prove that the rank of quadratic storage functions is nonincreasing in time, allowing us to introduce a novel null space decomposition, under much weaker assumptions than the one needed for general matrix functions.
  Additionally, we show a necessary kernel condition for the quadratic storage function, allowing us to reduce our search scope even further.
\end{abstract}

\footnotetext[1]{%
Institute of Analysis und Numerics, Otto von Guericke University Magdeburg, Universit\"atsplatz 2, 39106 Magdeburg. \texttt{riccardo.morandin@ovgu.de}.
The author gratefully acknowledges the support by Deutsche Forschungsgemeinschaft (DFG) through project 446856041.}

\footnotetext[2]{%
Institut f\"ur Mathematik, MA 4-5, TU Berlin, Str. des 17. Juni 136,
D-10623 Berlin, FRG.
\texttt{{hinsen}@math.tu-berlin.de}. The author gratefully acknowledges the support by the excellence cluster \textsf{MATH\textsuperscript{+}} and \textsf{BMBF(grant no.~05M22KTB)} through \textsf{EKSSE}.
}

\section{Introduction}

Dissipativity theory is extensively employed in the analysis and design of control systems \cite{LozBEM00}.
While it can be expressed in several different forms (e.g.~port-Hamiltonian systems \cite{SchJ14}, Kalman-Yakubovich-Popov inequality \cite{GusL06}, passivity \cite{Wil72}, positive realness \cite{Bru31}), the common idea is to focus on certain input-output properties of the system and their relation to the conservation, dissipation and transport of some energy quantity, typically known as total energy (in the case of physics), Hamiltonian (in the case of port-Hamiltonian systems), storage function (in the case of passive systems), or Lyapunov function. In this paper we focus on passive systems, thus we call such a functional ``storage function'', consistently with the nomenclature used in control theory.

Although these concepts are deeply understood in the case of linear time-invariant (LTI) systems, the literature for linear time-varying (LTV) systems is rather scarce. 
However, when working with general systems, a fundamental step is to consider LTV systems.
For example, LTV systems can emerge from the linearization of nonlinear dynamical systems along a reference solution \cite{Cam95}, or by approximating hard-to-model nonlinear systems with parametrized linear time-invariant systems, whose parameters vary with time. 
Quite strong regularity and controllability assumptions are often taken to simplify the treatment of problems.

In this paper, we state results for more general assumptions. We assume very mild regularity on the coefficients of the LTV systems, with the minimal requirement of ensuring the existence and uniqueness of weak solutions following from the Carath\'eodory conditions \cite{Fil88}.
Even in the LTI case, constructing a storage function is not a trivial matter. Finding a storage function in the general case often involves solving a Lyapunov equation and a lot of computational work \cite{DieV02}.

In this paper, we offer valuable knowledge to simplify the search for a storage function in the LTV case. Crucially, we show that every passive LTV system has at least one quadratic storage function, which coincides with the available storage and is therefore minimal among all storage functions.

Moreover, we introduce an extension of the well-known null space decomposition results from \cite{Dol64} and \cite[Theorem 3.1.9]{KunM24}. The null space decomposition in this paper holds for any generic weakly decreasing time-dependent pointwise Hermitian positive semidefinite matrix function (in the sense of the Loewner ordering), and extends to the case of quadratic storage functions, without requiring any rank restriction or additional regularity condition.
This allows us to split this matrix function into energy-storing parts and energy-independent parts. Furthermore, we present a necessary kernel condition for the quadratic storage function to ease the search for a storage function.

The paper is structured in the following way. We start by introducing functions of bounded variation and absolutely semicontinuous functions as preliminary results in \Cref{sec: preliminaries}.
We then focus on passive LTV systems' properties and their corresponding time-varying storage functions.
In \Cref{sec: regularity}, we examine the regularity properties of quadratic storage functions.
In \Cref{sec:availableStorage}, we take a closer look at the available storage for LTV systems and introduce equivalent conditions connecting the available storage with the passivity of LTV systems. This helps us to show that every passive LTV system must admit a time-varying storage function that is quadratic and positive semidefinite in the state variable.
Then, in \Cref{sec: null space decomposition}, we gain a deeper understanding of the necessary rank properties of the quadratic storage functions, which provides us with a null space decomposition, first for weakly decreasing pointwise Hermitian positive semidefinite matrices, then for the matrix functions inducing quadratic storage functions for LTV systems. This brings us, in \Cref{sec:kernel result}, to extend a necessary condition for the kernel of the matrix function $Q$, which defines the quadratic storage function of a passive LTV system, a result that is well-known for LTI port-Hamiltonian systems \cite{Sch09}.

The null space decomposition of the storage function will be especially relevant in a follow-up work \cite{CheGHMM24}, where we analyze other dissipativity concepts for LTV systems and their relations with passivity and among themselves.

\subsection{Notation}

We denote by $\N$ the set of positive natural numbers and by $\N_0$ the natural numbers including zero.
For any function $f:X\to Z$ and subset $Y\subseteq X$ we denote by $\restr{f}{Y}:Y\to Z$ the restriction of $f$ to the smaller domain $Y$.
If $X\subseteq\R$, $f:X\to Y$, and $x_0\in X$ is such that $(x_0-\varepsilon,x_0)\subseteq X$ for small enough $\varepsilon>0$, then we denote by $\lim_{x\to x^-}f(x)$ the \emph{left-hand limit} of $f$ in $x$, whenever it exists. We denote by $\lim_{x\to x^+}f(x)$ the corresponding \emph{right-hand limit}, under analogous conditions.
We recall that in general $\lim_{x\to x^-}f(x)$, $\lim_{x\to x}f(x)$, and $\lim_{x\to x^+}f(x)$ can be different numbers, unless $f$ is continuous.

For every complex vector $v\in\C^n$ and complex matrix $M\in\C^{m,n}$, we denote by $\norm{v}_2$ and $\norm{M}_2$ their 2-norm and spectral norm, respectively, and by $\ct{v}\in\C^{1,n}$ and $\ct{M}\in\C^{n,m}$ the corresponding conjugate transposes, where we identified $\C^n\cong\C^{n,1}$ in the natural way. 
For $n\in\N$, we denote by $\GL[n]\subseteq\C^{n,n}$ the multiplicative group of invertible matrices of size $n$, by $\HerMat[n]\subseteq\C^{n,n}$ the linear subspace of Hermitian matrices, by $\posSD[n]\subseteq\HerMat[n]$ the subset of positive semi-definite matrices, and by $\posDef[n]=\GL[n]\cap\posSD[n]$ the subset of positive definite matrices.
If $M\in\posSD[n]$ (resp.~$M\in\posDef[n]$), we also write $M\geq 0$ (resp.~$M>0$).
Furthermore, we equip $\HerMat[n]$ with the Loewner partial order, i.e., for every $M_1,M_2\in\HerMat[n]$ we write $M_1\geq M_2$ (resp.~$M_1>M_2$) if $M_1-M_2\geq 0$ (resp.~$M_1-M_2>0$).

For any open set $\Omega\subseteq\R^d$ we denote by $\mathcal C(\Omega,\R)$ the continuous maps from $\Omega$ to $\R$, and by $\mathcal C^k(\Omega,\R)$ for $k\in\N_0\cup\set{\infty}$ the corresponding $k$-times continuously differentiable maps. The notation generalizes straightforwardly to matrix-valued maps by applying it entrywise, and to complex-valued maps by identifying $\C\cong\R^2$ (thus differentiability is \emph{not} to be intended in the complex sense).

For any Lebesgue measurable set $\Omega\subseteq\R^d$, we denote by $\abs{\Omega}\in[0,+\infty]$ its measure.
For $p\in[1,\infty]$ we denote by $L^p(\Omega,\C)$ the usual Lebesgue spaces and by $W^{k,p}(\Omega,\C)$ for $k\in\mathbb N$ and $\Omega$ open the corresponding Sobolev spaces, see e.g.~\cite{Bre10}.
Once again, differentiability is not to be intended in the complex, holomorphic sense, but in the real sense, identifying $\C\cong\R^2$.
Moreover, we define $L^p(\Omega,\C^{n})$ (and analogously $L^p(\Omega,\C^{m,n})$ and $W^{k,p}(\Omega,\C^n)$) as the space of vector-valued functions whose entries are in $L^p(\Omega,\C)$.
In particular, we equip $L^p(\Omega,\C^n)$ with the norm
\begin{alignat*}{3}
	& \norm{\,\cdot\,}_{L^p} &&: L^p(\Omega,\C^n) \to \R, &\qquad f &\mapsto \norm[\big]{ \norm{f}_2 }_{L^p}, \\
	& \norm{f}_2 &&: \Omega \to \R, &\qquad \omega &\mapsto \norm{f(\omega)}_2,
\end{alignat*}
and we analogously define norms on $L^p(\Omega,\C^{m,n})$ and $W^{k,p}(\Omega,\C^n)$.
We note that the aforementioned Lebesgue and Sobolev spaces are Banach spaces, and that $L^2(\Omega,\C^{n,1})$ is a Hilbert space with respect to the inner product.
\[
\aset{f,g}_{L^2} = \int_{\Omega}\ct{g(\omega)}f(\omega)\,\td \omega, \qquad\text{for all }f,g\in L^2(\Omega,\C^n).
\]
We also introduce the local Lebesgue space \label{glo:Lploc}
\[
L^p_\loc(\Omega,\C) = \set[\big]{ f : \Omega\to\C \mid f|_K \in L^p(\Omega,\C)\text{ for all compact subsets }K\subseteq\Omega }.
\]
Sometimes we omit the domain and codomain from the function space notation when they are clear from the context.
While $L^p_\loc$ and $W^{k,p}_\loc$ are in general not even equipped with a norm, functions in these spaces can be reinterpreted as functions in $L^p$ and $W^{k,p}$, up to restriction to an appropriate compact subset.
Most of the results in this paper can also be applied for real-valued solutions and coefficients, and for $L^p$ and $W^{k,p}$ instead of their local variants.

If $A:\Omega\to\GL[n]$ is a pointwise invertible matrix function, we denote by
$
A^{-1} : \Omega \to \GL[n], \ \omega \mapsto A(\omega)^{-1}
$
the pointwise inverse matrix function.
Note that, if $A$ is defined and invertible only for a.e.~$\omega\in\Omega$, e.g.~if $A\in L^p_\loc(\Omega,\GL[n])$, then $A^{-1}$ is only defined a.e.~on $\omega\in\Omega$.

\subsection{Setting}\label{subsec:setting}

We consider linear time-varying systems of the form 
\begin{equation}
\label{eq:tv_system}
\begin{split}
    \dot x(t) &= A(t)x(t)+B(t)u(t),\\
    y(t) &= C(t)x(t)+D(t)u(t),
\end{split}
\end{equation}
on a possibly unbounded open time interval $\timeInt\subseteq\mathbb R$.
Here $x\in W^{1,1}_\loc(\timeInt,\C^n)$ is the state variable, $u,y\in ~L^2_\loc(\timeInt,\C^m)$ are the input and output variables, and $A\in L^1_\loc(\timeInt,\C^{n,n})$, $B\in L^2_\loc(\timeInt,\C^{n,m})$, $C\in L^2_\loc(\timeInt,\C^{m,n})$, and $D\in L^\infty_\loc(\timeInt,\C^{m,m})$ are matrix functions.

Note that, since the time domain is 1-dimensional, it holds that $W^{1,1}_\loc(\timeInt,\C^n)\subseteq\mathcal C(\timeInt,\C^n)\subseteq L^\infty_\loc(\timeInt,\C^n)$ (up to switching to the appropriate representative). Furthermore, for every input $u\in L^2_\loc(\timeInt,\C^m)$, the system \eqref{eq:tv_system} satisfies the Carath\'eodory conditions for the existence and uniqueness of solutions (see also \Cref{cor:solution_tv}). Thus, for every pair $(t_0,x_0)\in\timeInt\times\C^n$ and $u\in L^2_\loc(\timeInt,\C^m)$, there exists exactly one state solution $x\in W^{1,1}_\loc(\timeInt,\C^n)$ satisfying $x(t_0)=x_0$.
Moreover, the choice of function spaces for $C$ and $D$ ensures that $y\in L^2_\loc(\timeInt,\C^m)$, because of \Cref{thm:genHolder}.
We call a triple $(x,u,y)\in W^{1,1}_\loc(\timeInt,\C^n)\times L^2_\loc(\timeInt,\C^m)\times L^2_\loc(\timeInt,\C^m)$ that satisfies \eqref{eq:tv_system} for a.e.~$t\in\timeInt$ a \emph{state-input-output solution}.

We are particularly interested in systems which are passive in the sense introduced by Willems for LTI systems \cite{Wil72}. 
The following definition, introduced e.g.~in \cite{HilM80}, generalizes passivity to the case of time-varying systems.

\begin{definition}\label{def:passive}
    We say that $V:\timeInt\times\C^n\to\R$ is a \emph{storage function} for the LTV system \eqref{eq:tv_system} if it satisfies the following properties:
    \begin{enumerate}
        \item $V(t,x)\geq 0$ for all $t\in\timeInt$ and $x\in\C^n$,
        \item $V(t,0)=0$ for all $t\in\timeInt$,
        \item The \emph{dissipation inequality}
        \begin{equation}\label{eq:dissIneq}
            V\pset[\big]{t_1,x(t_1)} - V\pset[\big]{t_0,x(t_0)} \leq \int_{t_0}^{t_1}\realPart\pset[\big]{\ct{y(t)}u(t)}\td t
        \end{equation}
        holds for all $t_0,t_1\in\timeInt,\ t_0\leq t_1$, along every state-input-output solution $(x,u,y)$ of \eqref{eq:tv_system}.
    \end{enumerate}
    We call an LTV system \emph{passive} if it admits a storage function.
\end{definition}

\noindent Sometimes passivity defined as in \Cref{def:passive} is referred to as \emph{impedance passivity}, to distinguish it from definitions that use a different supply rate (which would replace $\ct{y}u$ in \eqref{eq:dissIneq}).

In this paper we focus on the properties of state-quadratic time-varying storage functions (from now on simply referred to as \emph{quadratic storage functions}), i.e., maps of the form
\begin{equation}\label{eq:quadStorFunc}
\quadSt{Q} : \timeInt\times\C^n\to\R, \qquad (t,x) \mapsto \frac{1}{2}\ct{x}Q(t)x,
\end{equation}
for some fixed $Q:\timeInt\to\posSD[n]$. As we will show in \Cref{sec:availableStorage}, passive LTV systems always admit at least one quadratic storage function.

We stress that in this paper we often switch between considering functions over the open time interval $\timeInt$ and over a compact time interval $[t_0,t_1]$.
This is due to the fact that our function spaces of choice for \eqref{eq:tv_system} are local function spaces over $\timeInt$, so many properties have to be verified with respect to their restriction to compact subintervals $[t_0,t_1]\subseteq\timeInt$.

\section{Preliminaries} \label{sec: preliminaries}

\subsection{The fundamental solution matrix}

Let us consider the homogeneous differential equation
\begin{equation}\label{eq:homODE}
    \dot x(t) = A(t)x(t)
\end{equation}
for some $A\in L^1_\loc(\timeInt,\C^{n,n})$, which corresponds to \eqref{eq:tv_system} when fixing $u=0$.
As mentioned in \Cref{subsec:setting}, for every initial condition $t_0\in\timeInt,\ x_0\in\C^n$ there is a unique solution $x\in W^{1,1}_\loc(\timeInt,\C^n)$ satisfying $x(t_0)=x_0$.
One convenient way to express $x$ as a function of $x_0$ is by using the \emph{fundamental solution matrix} $X\in W^{1,1}_\loc(\timeInt,\GL[n])$, defined as the solution of the homogeneous matrix differential equation
\[
\dot X(t)=A(t)X(t), \qquad X(t_0)=I_n,
\]
see also \Cref{thm:fundamentalSolution}.
In particular, we can then write $x(t)=X(t)x_0$ for all $t\in\timeInt$ and $x_0\in\C^n$. Note that it is still necessary to fix an initial time $t_0\in\timeInt$.

\subsection{Functions of bounded variation}

In this section, we recall the class of functions of bounded variation, which will be used for studying the regularity of time-varying storage functions.

Given any nonempty compact interval $[t_0,t_1]\subseteq\timeInt$, we denote by
\[
\partInt{[t_0,t_1]} = \set[\big]{ \partit = \set{\wt t_0,\wt t_1,\ldots,\wt t_K} \mid t_0 = \wt t_0 < \wt t_1 < \ldots < \wt t_K = t_1,\ K\in\N }
\]
the set of \emph{partitions} of $[t_0,t_1]$. Given a partition $\partit=\set{\wt t_0,\wt t_1,\ldots,\wt t_K}\in\partInt{[t_0,t_1]}$, we denote by $\norm{\partit}\coloneqq \sup_{1\leq k\leq K}(\wt t_k-\wt t_{k-1})$ its \emph{norm}.
To keep the notation brief, we often write $\partit\in\partInt{}$ as shorthand for
$\partit=\set{\wt t_0,\wt t_1,\ldots,\wt t_K}\in\partInt{[t_0,t_1]}$.
In particular, given a function $f:[t_0,t_1]\to\C$, we denote its \emph{total variation} as
\[
\totVar[t_0][t_1](f) \coloneqq \sup_{\partit\in\partInt{}} \sum_{k=1}^K \abs{ f(\wt t_k) - f(\wt t_{k-1}) }.
\]
In particular, we say that $f:[t_0,t_1]\to\C$ is a function of bounded variation, or \emph{BV function}, if $\totVar[t_0][t_1](f)<\infty$, and we denote by $\BV([t_0,t_1],\C)$ the vector space of BV functions on $[t_0,t_1]$.
Furthermore, we say that $f:\timeInt\to\C$ is a function of locally bounded variation, or \emph{locally BV function}, if $\restr{f}{[t_0,t_1]}:[t_0,t_1]\to\C$ is a BV function for all $t_0,t_1\in\timeInt,\ t_0\leq t_1$, and we denote by $\BV_\loc(\timeInt,\C)$ the vector space of locally BV functions on $\timeInt$.

Analogously, given a matrix function $F:[t_0,t_1]\to\C^{m,n}$, we define its \emph{total variation} as
\[
\totVar[t_0][t_1](F) \coloneqq \sup_{\partit\in\partInt{}} \sum_{k=1}^K \norm{ F(\wt t_k) - F(\wt t_{k-1}) },
\]
and we say that $F$ is a matrix function of bounded variation, or \emph{BV matrix function}, if $\totVar[t_0][t_1](F)<\infty$.
We define analogously to before the matrix functions of locally bounded variation, or \emph{locally BV matrix functions}, and the corresponding vector spaces $\BV([t_0,t_1],\C^{m,n})$ and $\BV_\loc(\timeInt,\C^{m,n})$.
Note that, since for every matrix function $F=[f_{ij}]:[t_0,t_1]\to\C^{m,n}$ and partition $\partit\in\partInt{}$ it holds that
\[
\sup_{i,j}\sum_{k=1}^K \abs{ f_{ij}(\wt t_k) - f_{ij}(\wt t_{k-1})} \leq \sum_{k=1}^K \norm{ F(\wt t_k) - F(\wt t_{k-1}) } \leq \sum_{i=1}^m \sum_{j=1}^n\sum_{k=1}^K\abs{ f_{ij}(\wt t_k) - f_{ij}(\wt t_{k-1})},
\]
a matrix function is (locally) BV if and only if all of its entries are (locally) BV.

\begin{remark}
    Let $t_0,t_1\in\timeInt,\ t_0\leq t_1$.
    Jordan's decomposition theorem \cite[Corollary 13.6]{Car00} states that every real-valued BV function $f:[t_0,t_1]\to\R$ can be written as $f=g-h$, where $g,h:[t_0,t_1]\to\R$ are monotonically non-decreasing functions.
    Many other properties follow: for example, $f$ is quasi-continuous, i.e., its left- and right-hand limits $\lim_{t\to\wt t^-}f(t)$ and $\lim_{t\to\wh t^+}f(t)$ exist and are finite at each point $\wt t\in(t_0,t_1]$ and $\wh t\in[t_0,t_1)$; equivalently, functions of bounded variation only have jump discontinuities or removable discontinuities.
    Furthermore, $f$ has at most countably many discontinuities, it is bounded, Riemann integrable, and therefore Lebesgue measurable.
    Additionally, $f$ is differentiable at a.e.~$t\in[t_0,t_1]$, $\dot f\in L^1([t_0,t_1],\R)$, and $f$ can be uniquely split into $f=f_a+f_s$, where $f_a\in W^{1,1}([t_0,t_1],\R)$ with $f_a(t_0)=0$ and $\dot f_s=0$ a.e.; $f_a$ and $f_s$ are sometimes called the \emph{absolutely continuous part} and the \emph{singular part} of $f$, respectively.
    
    These properties extend entrywise to the case of complex-valued (locally) BV matrix functions (by identifying $\C\cong\R^2$).
    In particular, $\BV([t_0,t_1],\C^{n,n})\subseteq L^\infty([t_0,t_1],\C^{n,n})$ and $\BV_\loc(\timeInt,\C^{n,n})\subseteq L^\infty_\loc(\timeInt,\C^{n,n})$.
    Furthermore, given any $F\in W^{1,1}([t_0,t_1],\C^{m,n})$ and any partition $\partit\in\partInt{}$, it holds that
    \[
    \sum_{k=1}^K \norm{F(\wt t_k) - F(\wt t_{k-1})} = \sum_{k=1}^K \norm*{ \int_{\wt t_{k-1}}^{\wt t_k} \dot F(t)\td t } \leq \sum_{k=1}^K \int_{\wt t_{k-1}}^{\wt t_k} \norm{\dot F(t)}\td t  = \int_{t_0}^{t_1} \norm{\dot F(t)}\td t  \leq \norm{F}_{W^{1,1}},
    \]
    thus $\totVar[t_0][t_1](F)\leq\norm{F}_{W^{1,1}}$.
    In particular, we have that $W^{1,1}([t_0,t_1],\C^{m,n})\subseteq\BV([t_0,t_1],\C^{m,n})$ and $W^{1,1}_\loc(\timeInt,\C^{m,n})\subseteq\BV_\loc(\timeInt,\C^{m,n})$.
\end{remark} 

\subsection{Absolutely semicontinuous functions}

An intermediate regularity assumption that is stricter than bounded variation but weaker than absolute continuity is sometimes employed to study discontinuous differential equations.
Such functions appear in the literature with different names and equivalent definitions: upper (and lower) absolutely continuous, semi-absolutely continuous, absolute upper (and lower) semicontinuous, and of bounded variation with nonincreasing (and nondecreasing) singular part (see e.g.~\cite{Lee78,Pon77,Pou01,Top17}).
We combine here several of these equivalent definitions under one.
\begin{definition}\label{def:AUC_function}
    We call a real-valued function $f:[t_0,t_1]\to\R$ \emph{absolutely upper semicontinuous} and we write $f\in\AUC([t_0,t_1],\R)$ if any of the following equivalent definitions are satisfied:
    \begin{enumerate}[label=(\alph*)]
        \item For every $\varepsilon>0$ there is $\delta>0$ such that, for every choice of $r_1,s_1,\ldots,r_K,s_K\in[t_0,t_1]$ with $r_1<s_1\leq r_2<s_2\leq\ldots\leq r_K<s_K$, it holds that
        \begin{equation*}\label{eq:storageFundamentalInequality:proof:2}
            \sum_{k=1}^K\pset{s_k-r_k}<\delta \implies \sum_{k=1}^K\pset[\big]{f(s_k)-f(r_k)} < \varepsilon.
        \end{equation*}
        \item There exists $g\in L^1([t_0,t_1],\R)$ such that
        \begin{equation*}
            f(s) - f(r) \leq \int_r^s g(t)\td t
        \end{equation*}
        holds for all $r,s\in[t_0,t_1],\ r\leq s$.
        \item The derivative $\dot f(t)$ exists for a.e.~$t\in[t_0,t_1]$, $\dot f\in L^1([t_0,t_1],\R)$, and
        \begin{equation*}
            f(s) - f(r) \leq \int_r^s \dot f(t)\td t
        \end{equation*}
        holds for all $r,s\in[t_0,t_1],\ r\leq s$.
        \item $f\in\BV([t_0,t_1],\R)$ with monotonically nonincreasing singular part.
    \end{enumerate}
    We call a function $f:\timeInt\to\R$ \emph{locally absolutely upper semicontinuous} and we write $f\in\AUC_\loc(\timeInt,\R)$ if $f|_{[t_0,t_1]}\in\AUC([t_0,t_1],\R)$ for all $t_0,t_1\in\timeInt,\ t_0\leq t_1$.
\end{definition}
\noindent We extend now the concept of absolute upper semi-continuity to complex-valued pointwise Hermitian matrix functions. One possibility would be to apply the definition entrywise and identify $\C\cong\R^2$. However, to use the concept in the context of matrix inequalities in the Loewner ordering, we will proceed differently. 
\begin{definition}\label{def:weaklyMonotoneMatFunc}
    We say that $Q:\timeInt\to\HerMat[n]$ is \emph{weakly monotonically increasing}, or simply \emph{weakly increasing}, if $Q(t_0)\leq Q(t_1)$ for all $t_0,t_1\in\timeInt,\ t_0\leq t_1$.    
    Analogously, we say that $Q$ is \emph{weakly monotonically decreasing}, or simply \emph{weakly decreasing}, if $Q(t_0)\geq Q(t_1)$ for all $t_0,t_1\in\timeInt,\ t_0\leq t_1$.
\end{definition}
\noindent Using the concept of weakly monotonically decreasing matrix functions, we obtain the following result.
\begin{theorem}\label{thm:AUC_matrix}
    Let $Q:[t_0,t_1]\to\HerMat[n]$ be a pointwise Hermitian matrix function. Then the following statements are equivalent:
    \begin{enumerate}[label=\rm(\roman*)]
        \item \label{it:AUC_matrix:1}
        The function $\quadSt{Q}^x:[t_0,t_1]\to\R,\ t\mapsto\frac{1}{2}\ct{x}Q(t)x$ is absolutely upper semicontinuous for all $x\in\C^n$.
        \item \label{it:AUC_matrix:1.5}
        For every $\varepsilon>0$ there exists $\delta>0$ such that, for every choice of $r_1,s_1,\ldots,r_K,s_K\in[t_0,t_1]$ with $r_1<s_1\leq r_2<s_2\leq\ldots\leq r_K<s_K$, it holds that
        \begin{equation*}\label{eq:storageFundamentalInequality:proof:2a}
            \sum_{k=1}^K\pset{s_k-r_k}<\delta \implies \sum_{k=1}^K\pset[\big]{Q(s_k)-Q(r_k)} < \varepsilon I_n.
        \end{equation*}
        \item \label{it:AUC_matrix:2}
        There exists $G\in L^1([t_0,t_1],\HerMat[n])$ such that
        \begin{equation*}
            Q(s) - Q(r) \leq \int_r^s G(t)\td t
        \end{equation*}
        holds for all $r,s\in[t_0,t_1],\ r\leq s$.
        \item \label{it:AUC_matrix:3}
        The derivative $\dot Q(t)$ exists for a.e.~$t\in[t_0,t_1]$, $\dot Q\in L^1([t_0,t_1],\HerMat[n])$ and the matrix inequality
        \begin{equation*}
            Q(s) - Q(r) \leq \int_r^s \dot Q(t)\td t
        \end{equation*}
        holds for all $r,s\in[t_0,t_1],\ r\leq s$.
        \item \label{it:AUC_matrix:4}
        $Q\in\BV([t_0,t_1],\HerMat[n])$ with weakly monotonically decreasing singular part.
    \end{enumerate}
\end{theorem}
\begin{proof}
    Note that, if $Q\in\BV([t_0,t_1],\HerMat[n])$, then $\dot Q\in L^1([t_0,t_1],\HerMat[n])$ and the splitting of $Q$ into its absolutely continuous part and its singular part can be explicitly constructed as
    \[
    Q_a : [t_0,t_1] \to \C^{n,n}, \quad t\mapsto\int_{t_0}^t\dot Q(s)\td s\qquad\text{and}\qquad Q_s \coloneqq Q - Q_a,
    \]
and thus $Q_a$ and $Q_s$ are also pointwise Hermitian.
    \begin{description}
        \item[\ref{it:AUC_matrix:4}$\implies$\ref{it:AUC_matrix:3}]
        Since $Q$ is of bounded variation, we have $\dot Q\in L^1([t_0,t_1],\HerMat[n])$, and we can split $Q=Q_a+Q_s$ with $Q_a$ absolutely continuous part and $Q_s$ singular part.
        For every $r,s\in[t_0,t_1],\ r\leq s$ we have then
        \[
        Q(s) - Q(r) = Q_a(s) - Q_a(r) + Q_s(s) - Q_s(r) \leq \int_{r}^s \dot Q_a(t)\td t = \int_r^s \dot Q(t)\td t.
        \]
        \item[\ref{it:AUC_matrix:3}$\implies$\ref{it:AUC_matrix:2}]
        This implication is obvious by taking $G=\dot Q$.
        \item[\ref{it:AUC_matrix:2}$\implies$\ref{it:AUC_matrix:1.5}]
        Let $\mathbf{G}:\timeInt\to\HerMat[n],\ t\mapsto\int_{t_0}^t G(t)\td t$, in particular $\mathbf G\in W^{1,1}([t_0,t_1],\HerMat[n])$. Then for every $\varepsilon>0$ there exist $\delta>0$ such that
        \[
        \sum_{k=1}^K\pset{s_k-r_k}<\delta \implies \sum_{k=1}^K\norm{\mathbf G(s_k)-\mathbf G(r_k)} < \varepsilon,
        \]
        since functions in $W^{1,1}([t_0,t_1],\R)$ are absolutely continuous. We deduce that
        \begin{align*}
            \sum_{k=1}^K\pset[\big]{Q(s_k)-Q(r_k)}
            &\leq \sum_{k=1}^K \int_{r_k}^{s_k}G(t)\td t
            = \sum_{k=1}^K \pset[\big]{\mathbf{G}(s_k)-\mathbf{G}(r_k)} \\
            &\leq \norm*{\sum_{k=1}^K \pset[\big]{\mathbf{G}(s_k)-\mathbf{G}(r_k)}}\cdot I_n
            \leq \sum_{k=1}^K\norm{\mathbf{G}(s_k)-\mathbf{G}(r_k)}\cdot I_n
            < \varepsilon I_n.
        \end{align*}
        \item[\ref{it:AUC_matrix:1.5}$\implies$\ref{it:AUC_matrix:1}]
        For $x=0$ the function $\quadSt{Q}^x=0$ is trivially absolutely semicontinuous. Let now $x\neq 0$ be fixed, let $\varepsilon>0$, and let $\delta>0$ be such that
        \[
        \sum_{k=1}^K\pset{s_k-r_k}<\delta \implies \sum_{k=1}^K\pset[\big]{Q(s_k)-Q(r_k)} < \frac{2\varepsilon}{\norm{x}^2} I_n.
        \]
        Then it follows that
        \begin{align*}
            \sum_{k=1}^K\pset{s_k-r_k}<\delta \implies
            \sum_{k=1}^K\pset[\big]{\quadSt{Q}^x(s_k)-\quadSt{Q}^x(r_k)}
            &= \frac{1}{2}\ct{x}\sum_{k=1}^K\pset[\big]{Q(s_k)-Q(r_k)}x \\
            &\leq \frac{1}{2}\ct{x}\pset*{\frac{2\varepsilon}{\norm{x}^2}I_n}x = \varepsilon.
        \end{align*}
        \item[\ref{it:AUC_matrix:1}$\implies$\ref{it:AUC_matrix:4}]
        We first show that $Q\in\BV([t_0,t_1],\C^{n,n})$. Consider entrywise 
        $Q=[q_{ij}]$ with $q_{ij}:[t_0,t_1]\to\C$. Since $Q$ is pointwise Hermitian, we have
        \[
        q_{ij} = \frac{1}{4}\pset*{\quadSt{Q}^{e_i+e_j}-\quadSt{Q}^{e_i-e_j}) + \frac{\imagUnit}{4}(\quadSt{Q}^{e_i-\imagUnit e_j}-\quadSt{Q}^{e_i+\imagUnit e_j}},
        \]
        for every $i,j$, where $e_i,e_j\in\C^n$ denote the corresponding vectors of the standard basis of $\C^n$, and $\imagUnit\in\C$ denotes the imaginary unit.
        Since $\quadSt{Q}^x\in\AUC([t_0,t_1],\R)\subseteq\BV([t_0,t_1],\R)$ for all $x\in\C^n$, we deduce that $q_{ij}\in\BV([t_0,t_1],\C)$ for every $i,j$, and therefore $Q\in\BV([t_0,t_1],\C^{n,n})$.

        Let us now split $Q=Q_a+Q_s$ with $Q_a$ absolutely continuous part and $Q_s$ singular part: we need to show that $Q_s(s)\leq Q_s(r)$ for all $r,s\in[t_0,t_1],\ r\leq s$, or equivalently $\ct{x}Q_s(s)x\leq\ct{x}Q_s(r)x$ for all $x\in\C^n$.
        For fixed $x\in\C^n$, it is clear that $\quadSt{Q}^x=\quadSt{Q_a}^x+\quadSt{Q_s}^x$ is the splitting of $\quadSt{Q}^x$ into its absolutely continuous part and its singular part.
        In particular, $\quadSt{Q_s}^x$ is monotonically nonincreasing, i.e., $\ct{x}Q_s(s)x\leq\ct{x}Q_s(r)x$ for all $r,s\in[t_0,t_1],\ r\leq s$, as desired.
        \qedhere
    \end{description}
\end{proof}

\noindent We then introduce the following definition.
\begin{definition}\label{def:AUC_matrix}
    We call a pointwise Hermitian matrix function $Q:[t_0,t_1]\to\HerMat[n]$ \emph{absolutely upper semicontinuous} and we write $Q\in\AUC([t_0,t_1],\HerMat[n])$ if any of the equivalent statements in \Cref{thm:AUC_matrix} are satisfied.
    We call a pointwise Hermitian matrix function $Q:\timeInt\to\HerMat[n]$ \emph{locally absolutely upper semicontinuous} and we write $Q\in\AUC_\loc(\timeInt,\HerMat[n])$ if $\restr{Q}{[t_0,t_1]}\in\AUC([t_0,t_1],\HerMat[n])$ for all $t_0,t_1\in\timeInt,\ t_0\leq t_1$.
\end{definition}

\noindent Note that for $n=1$ \Cref{def:AUC_matrix} is equivalent to \Cref{def:AUC_function}, since $\HerMat[1]=\R$.
Furthermore, from the definition we immediately deduce that the inclusion chains
\begin{align*}
    & W^{1,1}([t_0,t_1],\HerMat[n])\subseteq\AUC([t_0,t_1],\HerMat[n])\subseteq\BV([t_0,t_1],\HerMat[n]) \subseteq L^\infty([t_0,t_1],\HerMat[n]) \qquad\text{and} \\
    & W^{1,1}_\loc(\timeInt,\HerMat[n])\subseteq\AUC_\loc(\timeInt,\HerMat[n])\subseteq\BV_\loc(\timeInt,\HerMat[n])\subseteq L^\infty_\loc(\timeInt,\HerMat[n])
\end{align*}
hold for all $t_0,t_1\in\timeInt,\ t_0\leq t_1$, and $n\in\N$.
Note that, if the domain intervals are not degenerate, all inclusions are strict, as the following example shows.
\begin{example}\label{exm:Cantor}
    A classical example of a function of bounded variation that is not absolutely continuous is the \emph{Cantor function} (see e.g.~\cite{Car00} for a precise definition).
    More precisely, the Cantor function $\Cantor:[0,1]\to[0,1]$ is weakly monotonic increasing from $\Cantor(0)=0$ to $\Cantor(1)=1$, in particular $\totVar[t_0][t_1](\Cantor)=1$, but its derivative is $\dot\Cantor=0$ a.e.~on $[0,1]$.
    In fact, $\Cantor\in\BV([0,1],\R)\setminus\AUC([0,1],\R)$, since $\Cantor$ coincides with its singular part, and it is not weakly monotonic decreasing.
    Note that, up to scaling or extending the definition interval appropriately, we can use $\Cantor\cdot I_n$ as an example for a function in $\BV([t_0,t_1],\posSD[n])\setminus\AUC([t_0,t_1],\posSD[n])$ or $\BV_\loc(\timeInt,\posSD[n])\setminus\AUC_\loc(\timeInt,\posSD[n])$.

    It is also sometimes useful to consider the \emph{reverse Cantor function} $\CantorReverse:[0,1]\to[0,1],\ t\mapsto\Cantor(1-t)$, which again is of bounded variation and coincides with its singular part, but is also weakly monotonic decreasing, thus $\CantorReverse\in\AUC([0,1],\R)\setminus W^{1,1}([0,1],\R)$. Analogously to before, one can use $\CantorReverse$  to construct examples of functions in $\AUC([t_0,t_1],\posSD[n])\setminus W^{1,1}([t_0,t_1],\posSD[n])$ and $\AUC_\loc(\timeInt,\posSD[n])\setminus W^{1,1}_\loc(\timeInt,\posSD[n])$.
\end{example}

\noindent Note that absolutely upper semicontinuous functions need not be continuous. However, since they are of bounded variation, their discontinuities can only be of the jump or removable type, and one can deduce from the semicontinuity that actually they can only be of the decreasing jump type. We see this more precisely in the case of pointwise Hermitian matrix functions.

\begin{lemma}\label{lem:AUC_jumps}
    Let $Q\in\AUC_\loc(\timeInt,\HerMat[n])$. Then
    \[
    \lim_{t\to t_0^-}Q(t) \geq Q(t_0) \geq \lim_{t\to t_0^+}Q(t)
    \]
    holds for all $t_0\in\timeInt$.
\end{lemma}
\begin{proof}
    Note that the left- and right-hand limits of $Q$ exist because $\AUC_\loc\subseteq\BV_\loc$.
    Let us split $Q=Q_a+Q_s$ into its absolutely continuous part and singular part, such that $Q_a(t_0)=0$.
    Since $Q_a$ is continuous and $Q_s$ is weakly monotonically decreasing, we have that
    \[
    \lim_{t\to t_0^-}Q(t) - Q(t_0) = \lim_{t\to t_0^-}Q_a(t) + \lim_{t\to t_0^-}Q_s(t) - Q_a(t_0) - Q_s(t_0) = \lim_{t\to t_0^-}Q_s(t) - Q_s(t_0) \geq 0,
    \]
and hence $\lim_{t\to t_0^-}Q(t)\geq Q(t_0)$. We obtain analogously $Q(t_0)\geq\lim_{t\to t_0^+}Q(t)$.
\end{proof}

\noindent The following lemma shows that absolutely upper semicontinuous matrix functions are closed under congruence by absolutely continuous matrix functions.
\begin{lemma}\label{lem:AUC_congruence}
    Let $Q\in\AUC([t_0,t_1],\HerMat[n])$ and $V\in W^{1,1}([t_0,t_1],\C^{n,m})$. Then $\wt Q\coloneqq\ct{V}QV\in\AUC([t_0,t_1],\HerMat[m])$.
    Furthermore, if $Q\in\AUC([t_0,t_1],\posSD[n])$, then $\wt Q\in\AUC([t_0,t_1],\posSD[m])$.
\end{lemma}

\begin{proof}
    It is clear by definition that $Q(t)\in\HerMat[n]\implies\wt Q(t)\in\HerMat[m]$ and analogously $Q(t)\in\posSD[n]\implies\wt Q(t)\in\posSD[n]$ for all $t\in[t_0,t_1]$.
    Let now $r,s\in[t_0,t_1],\ r\leq s$. Because of \Cref{thm:AUC_matrix}\ref{it:AUC_matrix:3} we have
    \begin{align*}
        \wt Q(s) - \wt Q(r) &= \ct{V(s)}Q(s)V(s) - \ct{V(r)}Q(r)V(r) \leq \\
        &\leq \ct{V(s)}Q(r)V(s) + \ct{V(s)}\pset*{\int_r^s\dot Q(t)\td t}V(s) - \ct{V(r)}Q(r)V(r) \leq \\
        &\leq \pset*{\norm*{ \ct{V(s)}Q(r)V(s) - \ct{V(r)}Q(r)V(r) } + \norm{V}_{L^\infty}^2\int_r^s\norm{\dot Q(t)}\td t}I_n.
    \end{align*}
    Since
    \begin{align*}
        & \norm*{ \ct{V(s)}Q(r)V(s) - \ct{V(r)}Q(r)V(r) } = \norm*{ \ct{\pset[\big]{V(s) - V(r)}}Q(r)V(s) + \ct{V(r)}Q(r)\pset[\big]{V(s)-V(r)} } = \\
        &\qquad= \norm*{ \ct{\pset*{\int_r^s\dot V(t)\td t}}Q(r)V(s) + \ct{V(r)}Q(r)\int_r^s\dot V(t)\td t }
        \leq 2\norm{Q}_{L^\infty}\norm{V}_{L^\infty}\int_r^s\norm{\dot V(t)}\td t,
    \end{align*}
    we obtain
    \[
    \wt Q(s) - \wt Q(r) \leq \int_r^s\norm{V}_{L^\infty}\pset*{ 2\norm{Q}_{L^\infty}\norm{\dot V(t)} + \norm{V}_{L^\infty}\norm{\dot Q(t)}}I_n\td t.
    \]
    By taking
    \[
    G : [t_0,t_1]\to\C^{n,n},\qquad t\mapsto\norm{V}_{L^\infty}\pset{ 2\norm{Q}_{L^\infty}\norm{\dot V(t)} + \norm{V}_{L^\infty}\norm{\dot Q(t)}}I_n,
    \]
    we deduce that $G\in L^1([t_0,t_1],\HerMat[n])$ does not depend on $r,s$, and that $\wt Q(s)-\wt Q(r)\leq\int_r^s G(t)\td t$.
    We conclude with \Cref{thm:AUC_matrix}\ref{it:AUC_matrix:2} that $\wt Q\in\AUC([t_0,t_1],\HerMat[n])$.
\end{proof}

\noindent Finally, we characterize weakly decreasing matrix functions in terms of their absolute semicontinuity and classical derivative.

\begin{theorem}\label{lem:BV_decreasing}
    Let $Q:\timeInt\to\HerMat[n]$. Then $Q$ is weakly monotonically decreasing if and only if $Q\in\AUC_\loc(\timeInt,\HerMat[n])$ and $\dot Q(t)\leq 0$ for a.e.~$t\in\timeInt$.
\end{theorem}

\begin{proof}
    If $Q\in\AUC_\loc(\timeInt,\HerMat[n])$ and $\dot Q(t)\leq 0$ for a.e.~$t\in\timeInt$, then $Q(s)\leq Q(r)$ for all $r,s\in\timeInt,\ r<s$, because of \Cref{thm:AUC_matrix}\ref{it:AUC_matrix:3}.

    Suppose now that $Q:\timeInt\to\HerMat[n]$ is weakly monotonically decreasing.
    In particular, by taking the constant zero matrix function $G\equiv 0\in L^1(\timeInt,\HerMat[n])$ in \Cref{thm:AUC_matrix}\ref{it:AUC_matrix:2}, we deduce that $Q\in\AUC_\loc(\timeInt,\HerMat[n])$.
    Let $\timeInt_d\subseteq\timeInt$ be the set of points where $Q$ has a classical derivative, thus $\abs{\timeInt\setminus\timeInt_d}=0$.
    Since $Q$ is weakly decreasing, for every $t\in\timeInt_d$ and $h>0$ we have $Q(t+h)-Q(t)\leq 0$, and therefore
    \[
    \dot Q(t) = \lim_{h\to 0}\frac{Q(t+h)-Q(t)}{h} = \lim_{h\to 0^+}\frac{Q(t+h)-Q(t)}{h} \leq 0
    \]
    holds.
\end{proof}

\noindent In this section we have summarized and extended some classical results to prepare for the subsequent analysis of the properties of storage functions.

\section{Regularity of quadratic storage functions}\label{sec: regularity}
In this section, we study the regularity of quadratic storage functions, as introduced in \eqref{eq:quadStorFunc}.

\begin{lemma}\label{lem:passiveInequality}
    Suppose that $V:\timeInt\times\C^n\to\R$ is a storage function for a passive LTV system of the form \eqref{eq:tv_system}, and
    let $X\in W^{1,1}_\loc(\timeInt,\GL[n])$ denote the fundamental solution matrix associated to $\dot x=Ax$ for some fixed initial time $t_0\in\timeInt$.
    Then, for every $x_0\in\C^n$ and $r,s\in\timeInt$ with $r\leq s$, the inequality
    \begin{equation}\label{eq:passiveInequality:1}
        V\pset[\big]{ s, X(s)x_0 } \leq V\pset[\big]{ r, X(r)x_0 }
    \end{equation}
    holds.
    Furthermore, if $V=\quadSt{Q}$ for some $Q:\timeInt\to\posSD[n]$, then the map
    \begin{equation*}\label{eq:passiveInequality:2}
    Q^X : \timeInt \to \posSD[n], \qquad t \mapsto \ct{X(t)}Q(t)X(t)
    \end{equation*}
    is weakly monotonically decreasing, in particular $Q^X\in\AUC_\loc(\timeInt,\posSD[n])$ with $\dot Q^X\leq 0$.
\end{lemma}

\begin{proof}
    The state trajectory $x:\timeInt\to\C^n,\ t\mapsto X(t)x_0$ is the unique solution of \eqref{eq:tv_system} with vanishing input signal $u=0$, such that $x(t_0)=x_0$.
    In particular, we obtain that
    \[
    V\pset[\big]{ s, X(s)x_0 } = V\pset[\big]{ s , x(s) } \leq V\pset[\big]{ r , x(r) } = V\pset[\big]{ r, X(r)x_0 }
    \]
    for all $r,s\in\timeInt,\ r\leq s$.
    Suppose now that $V=\quadSt{Q}$ with $Q:\timeInt\to\posSD[n]$.
    It is clear that $Q^{X}$ is well-defined. Furthermore, because of \eqref{eq:passiveInequality:1} we have that
    \[
    \ct{x_0}\pset[\big]{ Q^X(s) - Q^X(r) }x_0
    = 2V\pset[\big]{ s, X(s)x_0 }-2V\pset[\big]{ r, X(r)x_0 } \leq 0
    \]
    for all $r,s\in\timeInt,\ r\leq s$ and $x_0\in\C^n$, thus $Q^X(s)\leq Q^X(r)$, i.e., $Q^X$ is weakly decreasing.
    Finally, by applying \Cref{lem:BV_decreasing}, we conclude that $Q^X\in\AUC_\loc(\timeInt,\posSD[n])$ with $\dot Q^X\leq 0$.
\end{proof}

\noindent We deduce the regularity of quadratic storage functions.

\begin{theorem}\label{thm:storageFunctionAUC}
    Suppose that, for some $Q:\timeInt\to\posSD[n]$, $\quadSt{Q}$ is a storage function for a passive LTV system of the form \eqref{eq:tv_system}. Then $Q\in\AUC_\loc(\timeInt,\posSD[n])$.
\end{theorem}
\begin{proof}
    Let $t_0,t_1\in\timeInt,\ t_0\leq t_1$.
    Because of \Cref{lem:passiveInequality}, it holds that $Q^X\in\AUC_\loc(\timeInt,\posSD[n])$.
    Let now $V:\timeInt\to\C^{n,n},\ t\mapsto X(t)^{-1}$, which is well-defined and in $W^{1,1}_\loc(\timeInt,\GL[n])$, because of \Cref{thm:fundamentalSolution}. Then
    \[
    \ct{V(t)}Q^X(t)V(t) = \ict{X(t)}\ct{X(t)}Q(t)X(t)X(t)^{-1} = Q(t)
    \]
    for all $t\in\timeInt$, i.e., $Q=\ct{V}Q^XV$.
    We conclude with \Cref{lem:AUC_congruence} that $Q\in\AUC_\loc(\timeInt,\posSD[n])$.
\end{proof}

    \noindent Let us recall that $L^p$ functions are identified when they coincide almost everywhere.
    Although the definitions of BV functions and absolutely upper semicontinuous functions could be adjusted analogously, this is more challenging when considering the matrix function $Q\in\AUC_\loc(\timeInt,\posSD[n])$ inducing a quadratic storage function, unless we require the dissipation inequality \eqref{eq:dissIneq} to only hold for a.e.~$t_0,t_1\in\timeInt,\ t_0\leq t_1$.
    
    To address this problem, we look for a uniquely determined representative $\wt Q\in\AUC_\loc(\timeInt,\posSD[n])$, among the matrix functions that coincide with $\wt Q$ almost everywhere, such that $\wt Q$ also induces a quadratic storage function for the same passive LTV system.
    Since in the points of continuity of $Q$ it is straightforward to leave the function as it is, we analyze our freedom in the points of discontinuity, which form a set of measure zero, and can only be of the decreasing jump type, because of \Cref{lem:AUC_jumps}.

\begin{lemma}\label{lem:storageFunctionDiscontinuities}
    Suppose that $\quadSt{Q}$ with $Q\in\AUC_\loc(\timeInt,\posSD[n])$ is a storage function for a passive LTV system \eqref{eq:tv_system}, and let $\wt Q:\timeInt\to\posSD[n]$ such that
    $
    \lim_{s\to t^-}Q(s) \geq \wt Q(t) \geq \lim_{s\to t^+}Q(s)
    $
    holds for all $t\in\timeInt$.
    Then $\quadSt{\wt Q}$ is a storage function for \eqref{eq:tv_system}.
\end{lemma}

\begin{proof}
    Let $(x,u,y)$ be any state-input-output solution of \eqref{eq:tv_system} and let $t_0,t_1\in\timeInt,\ t_0<t_1$, since for $t_0=t_1$ the dissipation inequality \eqref{eq:dissIneq} holds trivially for $\quadSt{\wt Q}$.
    For every $r,s\in\timeInt$ such that $t_0<r<s<t_1$ we have
    \[
    \frac{1}{2}\ct{x(s)}Q(s)x(s) - \frac{1}{2}\ct{x(r)}Q(r)x(r) =
    \quadSt{Q}\pset[\big]{s,x(s)} - \quadSt{Q}\pset[\big]{r,x(r)} \leq \int_r^s \realPart\pset[\big]{ \ct{y(t)}u(t) } \td t,
    \]
    thus in the limit for $r\to t_0^+$ and $s\to t_1^-$ we obtain
    \[
    \frac{1}{2}\ct{x(t_1)}\pset*{\lim_{s\to t_1^-}Q(s)}x(t_1) - \frac{1}{2}\ct{x(t_0)}\pset*{\lim_{r\to t_0^+}Q(r)}x(t_0) \leq \int_{t_0}^{t_1} \realPart\pset[\big]{ \ct{y(t)}u(t) } \td t.
    \]
    We conclude that
    \begin{align*}
        &\quadSt{\wt Q}\pset[\big]{t_1,x(t_1)} - \quadSt{\wt Q}\pset[\big]{t_0,x(t_0)}
        = \frac{1}{2}\ct{x(t_1)}\wt Q(t_1)x(t_1) - \frac{1}{2}\ct{x(t_0)}\wt Q(t_0)x(t_0) \\
        &\qquad\leq \frac{1}{2}\ct{x(t_1)}\pset*{\lim_{s\to t_1^-}Q(s)}x(t_1) - \frac{1}{2}\ct{x(t_0)}\pset*{\lim_{r\to t_0^+}Q(r)}x(t_0) \leq \int_{t_0}^{t_1} \realPart\pset[\big]{ \ct{y(t)}u(t) } \td t,
    \end{align*}
    i.e., $\quadSt{\wt Q}$ satisfies the dissipation inequality and is a storage function for \eqref{eq:tv_system}.
\end{proof}

\noindent An immediate consequence of \Cref{lem:storageFunctionDiscontinuities} is that, if $\quadSt{Q}$ is a storage function for \eqref{eq:tv_system}, then it admits an equivalent representative that is continuous from the right with respect to the time variable. We state this more precisely in the following theorem.

\begin{theorem}
    Let $\quadSt{Q}$ with $Q\in\AUC_\loc(\timeInt,\posSD[n])$ be a storage function for an LTV system \eqref{eq:tv_system}.
    Then the matrix function
    \[
    \wt Q:\timeInt\to\C^{n,n},\qquad t\mapsto\lim_{s\to t^+}Q(t)
    \]
    is such that $\wt Q=Q$ almost everywhere, $\wt Q(t)\leq Q(t)$ for all $t\in\timeInt$, and $\quadSt{\wt Q}$ is a storage function for \eqref{eq:tv_system}.
\end{theorem}

\begin{proof}
    Since $\wt Q$ is locally bounded variation, it is continuous almost everywhere, thus $\wt Q=Q$ almost everywhere by definition of $\wt Q$.
    Because of \Cref{lem:AUC_jumps}, we have $\wt Q(t)\leq Q(t)$ for all $t\in\timeInt$ and we can apply \Cref{lem:storageFunctionDiscontinuities}, concluding $\quadSt{\wt Q}$ is a storage function for \eqref{eq:tv_system}.
\end{proof}

\noindent After introducing these properties we can use them in the following sections.
\section{The available storage}\label{sec:availableStorage}

We recall the definition of available storage for LTV systems from \cite{HilM80,LozBEM00}.
\begin{definition}\label{def:avstor}
    The \emph{available storage} of the LTV system \eqref{eq:tv_system} at time $t_0\in\timeInt$ and state $x_0\in\C^n$ is
\begin{equation}\label{eq:availableStorage}
        \avSt(t_0,x_0)
        = \sup_{\avStArg{}}\left( -\int_{t_0}^{t_1}\realPart(\ct{y(t)}u(t))\td t\right )
        = -\inf_{\avStArg{}}\int_{t_0}^{t_1}\realPart(\ct{y(t)}u(t))\td t,
    \end{equation}
    where the supremum (or infimum) is taken among all times $t_1\in\timeInt,\ t_1\geq t_0$ and inputs $u\in L^2_\loc(\timeInt,\C^m)$, and $x\in W^{1,1}_\loc(\timeInt,\C^n)$ and $y\in L^2_\loc(\timeInt,\C^m)$ denote the state and output trajectories uniquely determined by $u$ and by the condition $x(t_0)=x_0$.
\end{definition}

\noindent Since the supremum in \Cref{def:avstor} is taken over a nonempty set, it is always well defined.
Furthermore, since for $t_1=t_0$ the integral in \eqref{eq:availableStorage} vanishes, the available storage is always nonnegative.
We can then intepret it as a possibly infinite-valued function $V_a:\timeInt\times\C^n\to[0,+\infty]$.
If we actually have $V_a(t_0,x_0)<+\infty$ for all $t_0\in\timeInt$ and $x_0\in\C^n$, we say that the available storage is finite.

The finiteness of the available storage is closely related to the passivity of the system. In fact, we have the following equivalent statements.

\begin{theorem}\label{thm:availableStorage}
    Consider an LTV system of the form \eqref{eq:tv_system}. Then the following statements are equivalent:
    \begin{enumerate}[label=\rm (\roman*)]
        \item The system is passive.
        \item There exists $\beta:\timeInt\times\C^n\to\R$ such that, for every $t_0,t_1\in\timeInt,\ t_0\leq t_1$ and state-input-output solution $(x,u,y)$, it holds that
        \begin{equation*}\label{eq:passiveAlt}
            \int_{t_0}^{t_1}\realPart\pset[\big]{\ct{y}(t)u(t)}\td t \geq \beta\pset[\big]{t_0,x(t_0)}.
        \end{equation*}
        \item The available storage function is finite.
        \item The available storage function is of the form $V_a=\quadSt{Q}$ for some matrix function $Q:\timeInt\to\posSD[n]$.
    \end{enumerate}
\end{theorem}
\begin{proof}
    \begin{description}
        \item[\rm(i)$\Rightarrow$(ii)] Let $V:\timeInt\times\C^n\to\R$ be any storage function for the system and let us define $\beta\coloneqq-V$.
        Then for every $t_0,t_1\in\timeInt,\ t_0\leq t_1$ and state-input-output solution $(x,u,y)$ it holds that
        \[
        \beta\pset[\big]{t_0,x(t_0)} = -V\pset[\big]{t_0,x(t_0)} = V\pset[\big]{t_1,x(t_1)} - V\pset[\big]{t_0,x(t_0)} - V\pset[\big]{t_1,x(t_1)} \leq \int_{t_0}^{t_1}\realPart\pset[\big]{\ct{y(t)}u(t)}\td t,
        \]
        because of the dissipation inequality \eqref{eq:dissIneq} and of the non-negativity of storage functions.
        %
        \item[\rm(ii)$\Rightarrow$(iii)] The available storage satisfies
        \[
        0\leq V_a(t_0,x_0)=\sup_{\substack{(t_1,u):t_1\geq t_0, \\ x(t_0)\geq x_0}}-\int_{t_0}^{t_1}\realPart\pset[\big]{\ct{y}(t)u(t)}\td t\leq-\beta(t_0,x_0)<\infty,
        \]
        for all $(t_0,x_0)\in\timeInt\times\C^n$, thus it is finite.
        \item[\rm(iii)$\Rightarrow$(iv)] Let us define for simplicity
        \[
        \mathcal V(t_0,t_1,x_0,u) \coloneqq \int_{t_0}^{t_1}\realPart\pset[\big]{\ct{y(t)}u(t)}\td t
        \]
        for all $t_0,t_1\in\timeInt$ with $t_0\leq t_1$, initial state $x_0\in\C^n$ and admissible input $u$, and let us write in short $V_a(t_0,x_0)=-\inf_{t_1,u}\mathcal V(t_0,t_1,x_0,u)$, where $t_1\in\timeInt,\ t_1\geq t_0$ and $u\in L^2_\loc(\timeInt,\C^m)$ are implicit in the notation.
        Note that for every $t_0,t_1\in\timeInt$ with $t_0\leq t_1$, initial state $x_0\in\C^n$, complex number $\lambda\in\C$, and pair of state-input-output solutions $(x,u,y),(\wt x,\wt u,\wt y)$ such that $x(t_0)=x_0$ and $\wt x(t_0)=\wt x_0$, we have
        \[
        \mathcal V(t_0,t_1,\lambda x_0,\lambda u) = \int_{t_0}^{t_1}\realPart\pset*{\ct{\pset[\big]{\lambda y(t)}}\lambda u(t)}\td t
        = \abs{\lambda}^2\cdot\mathcal V(t_0,t_1,x_0,u)
        \]
        and
        \begin{align*}
        &\mathcal V(t_0,t_1,x_0+\wt x_0,u+\wt u) + \mathcal V(t_0,t_1,x_0-\wt x_0,u-\wt u) \\
        &\qquad= \int_{t_0}^{t_1}\realPart\pset*{\ct{\pset[\big]{y(t)+\wt y(t)}}\pset[\big]{u(t)+\wt u(t)}}\td t + \int_{t_0}^{t_1}\realPart\pset*{\ct{\pset[\big]{y(t)-\wt y(t)}}\pset[\big]{u(t)-\wt u(t)}}\td t \\
        &\qquad= 2\int_{t_0}^{t_1}\realPart\pset[\big]{\ct{y(t)}u(t)}\td t + 2\int_{t_0}^{t_1}\realPart\pset[\big]{\ct{\wt y(t)}\wt u(t)}\td t
        = 2\mathcal V(t_0,t_1,x_0,u) + 2\mathcal V(t_0,t_1,\wt x_0,\wt u),
        \end{align*}
        since the state-input-output solutions of \eqref{eq:tv_system} form a linear subspace (see \Cref{lem:linearitySolutions}).

        We show first that
        \begin{equation}\label{eq:quadraticFormProp:1}
         V_a(t_0,\lambda x_0)=\abs{\lambda}^2 \cdot V_a(t_0,x_0)   
        \end{equation}
        holds for all $t_0\in\timeInt$, $\lambda\in\C$, and $x_0\in\C^n$.
        Suppose first that $\lambda\neq 0$: then for every $\varepsilon>0$ there are $\underline{t_1}\geq t_0$ and $\underline{u}\in L^2_\loc(\timeInt,\C^m)$ such that
        $\inf_{t_1,u}\mathcal V(t_0,t_1,\lambda x_0,u)>\mathcal V(t_0,\underline{t_1},\lambda x_0,\underline{u})-\varepsilon$.
        Furthermore, since $\lambda^{-1}\underline{u}\in L^2_\loc(\timeInt,\C^m)$, we obtain
        \begin{align*}
            V_a(t_0,\lambda x_0) &= -\inf_{t_1,u}\mathcal V(t_0,t_1,\lambda x_0,u) < \varepsilon - \mathcal V(t_0,\underline{t_1},\lambda x_0,\underline{u}) = \varepsilon - \abs{\lambda}^{2}\mathcal V(t_0,\underline{t_1},x_0,\lambda^{-1}\underline{u}) \\
            &\leq \varepsilon - \abs{\lambda}^{2}\inf_{t_1,u}\mathcal V(t_0,t_1,x_0,u) = \varepsilon + \abs{\lambda}^{2}V_a(t_0,x_0).
        \end{align*}
        Since the choice of $\varepsilon>0$ was arbitrary, we deduce that 
        $V_a(t_0,\lambda x_0)\leq\abs{\lambda}^{2}V_a(t_0,x_0)$ for every $\lambda\in\C\setminus\{0\}$.
        In particular, by applying this property twice, be obtain
        \[
        V_a(t_0,\lambda x_0) \leq \abs{\lambda}^2\cdot V_a(t_0,x_0) = \abs{\lambda}^2\cdot V_a(t_0,\lambda^{-1}\lambda x_0) \leq \abs{\lambda}^2\abs{\lambda^{-1}}^2 V_a(t_0,\lambda x_0) = V_a(t_0,\lambda x_0),
        \]
        and therefore $V_a(t_0,\lambda x_0)=\abs{\lambda}^2\cdot V_a(t_0,x_0)$, for all $\lambda\in\C\setminus\{0\}$.
        To prove \eqref{eq:quadraticFormProp:1} for $\lambda=0$, we simply note that
        \[
        V_a(t_0, 0) = V_a(t_0, 2\cdot 0) = 4V_a(t_0, 0) \implies V_a(t_0, 0) = 0.
        \]
        We now show that
        \begin{equation}\label{eq:quadraticFormProp:2}
                    V_a(t_0, x_0 + \wt x_0) + V_a(t_0,x_0-\wt x_0) = 2V_a(t_0,x_0) + 2V_a(t_0,\wt x_0)
                \end{equation}
        holds for all $t_0\in\timeInt$ and $x_0,\wt x_0\in\C^n$, so that we may conclude that $V_a(t_0,\cdot):\C^n\to\R$ is a quadratic form for every fixed $t_0$.
                Analogously as before, for every $\varepsilon>0$ there are ${t_1},{\wt t_1}\geq t_0$ and ${u},{\wt u}\in L^2_\loc(\timeInt,\C^m)$ such that
                $
                V_a(t_0,x_0) < \varepsilon - \mathcal V(t_0,{t_1},x_0,{u})
                $
                and
                $V_a(t_0,\wt x_0) < \varepsilon - \mathcal V(t_0,{\wt t_1},\wt x_0,{\wt u}).
                $
                Note that we can assume without loss of generality that ${t_1}={\wt t_1}$: if we had e.g.~${t_1}<{\wt t_1}$, we could replace $ u$ with
        \[
                {\wh u}(t) \coloneqq
                \begin{cases}
                     u(t) & \text{for }t\leq{t_1}, \\
                    0 & \text{otherwise},
                \end{cases}
        \]
        and observe that the corresponding outputs satisfy ${\wh y}(t)={y}(t)$ for all $t\leq{t_1}$, and therefore
        \[
                \mathcal V(t_0,{\wt t_1},x_0,{\wh u})
                = \int_{t_0}^{{\wt t_1}} \realPart\pset[\big]{\ct{{\wh y}(t)}{\wh u}(t)} \td t 
                = \int_{t_0}^{{t_1}} \realPart\pset[\big]{\ct{{y}(t)}{u}(t)} \td t + 0
                = \mathcal V(t_0,{t_1},x_0,{u}).
                \]
        Then it holds that
                \begin{align*}
                    & V_a(t_0,x_0+\wt x_0) + V_a(t_0,x_0-\wt x_0) \\
                    &\qquad\geq \mathcal V(t_0,t_1,x_0+\wt x_0,u+\wt u) + \mathcal V(t_0,t_1,x_0-\wt x_0,u-\wt u) \\
                    &\qquad= 2\mathcal V(t_0,t_1,x_0,u) + 2\mathcal V(t_0,t_1,\wt x_0,\wt u) > 2V_a(t_0,x_0) + 2V_a(t_0,\wt x_0) - 2\varepsilon.
                \end{align*}
        Since $\varepsilon>0$ can be taken arbitrarily small, we obtain
        \[
                V_a(t_0,x_0+\wt x_0) + V_a(t_0,x_0-\wt x_0) \geq 2V_a(t_0,x_0) + 2V_a(t_0,\wt x_0)
                \]
        for all $t_0\in\timeInt$ and $x_0,\wt x_0\in\C^n$.
        By applying this argument twice, we deduce that
         \begin{align*}
                    & 2V_a(t_0,x_0) + 2V_a(t_0,\wt x_0)
                    = 2\pset[\big]{ V_a(t_0,\tfrac{x_0+\wt x_0}{2}+\tfrac{x_0-\wt x_0}{2}) + V_a(t_0,\tfrac{x_0+\wt x_0}{2}-\tfrac{x_0-\wt x_0}{2}) } \\
                    &\qquad\geq 4\pset[\big]{ V_a(t_0,\tfrac{x_0+\wt x_0}{2}) + V_a(t_0,\tfrac{x_0-\wt x_0}{2}) } = V_a(t_0,x_0+\wt x_0) + V_a(t_0,x_0-\wt x_0),
                \end{align*}
        and therefore \eqref{eq:quadraticFormProp:2} holds.
        
        From \eqref{eq:quadraticFormProp:1} and \eqref{eq:quadraticFormProp:2} we conclude that $V_a$ is a quadratic form with respect to $x$, i.e.,
        \[
        V_a(t,x)=\frac{1}{2}\ct{x}Q(t)x
        \]
        for a uniquely determined $Q:\timeInt\to\HerMat[n]$, for every $t\in\timeInt$ and $x\in\C^n$.
        Furthermore, from the non-negativity of $V_a$ we immediately deduce that actually $Q:\timeInt\to\posSD[n]$.
        \item[\rm(iv)$\Rightarrow$(i)] From the form of $V_a$, it trivially follows that it is finite and satisfies $V_a(t,0)=0$ for all $t\in\timeInt$.
        Let now $t_0,t_1\in\timeInt$ with $t_0\leq t_1$, let $(x,u,y)$ be any state-input-output solution of \eqref{eq:tv_system}, and let
        \[
        \mathcal U \coloneqq \set{ \wt u\in L^2_\loc(\timeInt,\C^m) \mid \wt u(t)=u(t)\text{ for all }t\in[t_0,t_1] }.
        \]
        For every $t_2\in\timeInt,\ t_2\geq t_1$ and $\wt u\in\mathcal U$ it holds that
        \[
        \mathcal V\pset[\big]{t_0,t_2,x(t_0),\wt u} = \int_{t_0}^{t_1}\realPart\pset[\big]{\ct{y(t)}u(t)}\td t + \mathcal V\pset[\big]{t_1,t_2,x(t_1),\wt u},
        \]
        because of the causality of the system (i.e., the fact that the values of $x(t),y(t)$ for a fixed $t\in\timeInt,\ t\geq t_0$ are uniquely determined by $x(t_0)$ and the restriction $u|_{[t_0,t]}$, see e.g.~\cite{LozBEM00}).
        In particular, we deduce that
        \begin{align*}
            V_a\pset[\big]{t_0,x(t_0)}
            &= -\inf_{\substack{t_2\geq t_0,\\ \wt u\in L^2_\loc}}\mathcal V\pset[\big]{t_0,t_2,x(t_0),\wt u}
            \geq -\inf_{\substack{t_2\geq t_1,\\ \wt u\in\mathcal U}}\mathcal V\pset[\big]{t_0,t_2,x(t_0),\wt u} \\
            &= -\int_{t_0}^{t_1}\realPart\pset[\big]{\ct{y(t)}u(t)}\td t - \inf_{\substack{t_2\geq t_1,\\ \wt u\in\mathcal U}} \mathcal V\pset[\big]{t_1,t_2,x(t_1),\wt u} \\
            &= -\int_{t_0}^{t_1}\realPart\pset[\big]{\ct{y(t)}u(t)}\td t + V_a\pset[\big]{t_1,x(t_1)},
        \end{align*}
        where
        \[
        \inf_{\substack{t_2\geq t_1,\\ \wt u\in\mathcal U}} \mathcal V\pset[\big]{t_1,t_2,x(t_1),\wt u}
        = \inf_{\substack{t_2\geq t_1,\\ \wt u\in L^2_\loc}} \mathcal V\pset[\big]{t_1,t_2,x(t_1),\wt u}
        \]
        holds again because of causality, and therefore $V_a$ is a storage function. \qedhere 
    \end{description}
\end{proof}

\noindent Note that the proof of \Cref{thm:availableStorage} (iv)$\Rightarrow$(i) also shows that, if the LTV system \eqref{eq:tv_system} is passive, then the available storage is a storage function.
Not only this is a well-known fact, but it can also be observed that $V_a$ is minimal among all storage functions of \eqref{eq:tv_system} (see e.g.~\cite{HilM80}).
Furthermore, \Cref{thm:availableStorage} also adds that $V_a=V_Q$ for some $Q:\timeInt\to\posSD[n]$. Combining these remarks with \Cref{thm:storageFunctionAUC}, we deduce the following corollary.
\begin{corollary}\label{cor:availableStorageMinimumAndAUC}
    The available storage of a passive LTV system of the form \eqref{eq:tv_system} is a quadratic form $V_a=V_Q$ for some $Q\in\AUC_\loc(\timeInt,\posSD[n])$ and is minimal among all storage functions of \eqref{eq:tv_system}, in the sense that $V_a(t_0,x_0)\leq V(t_0,x_0)$ holds for every storage function $V:\timeInt\times\C^n\to\R$, $t_0\in\timeInt$ and $x_0\in\C^n$.
\end{corollary}
\noindent In particular, the following statement obviously holds.

\begin{corollary}\label{cor:Pa_to_quadraticStorage}
    Every passive LTV system of the form \eqref{eq:tv_system} has a quadratic storage function $\quadSt{Q}$ for some $Q\in\AUC_\loc(\timeInt,\posSD[n])$.
\end{corollary}

\noindent We stress that \Cref{cor:Pa_to_quadraticStorage} is an extremely relevant result, since it allows us to restrict the search of storage functions to the relatively contained class of quadratic forms.

\section{Null space decomposition}\label{sec: null space decomposition}

By \emph{null space decomposition} of a generic matrix function $B:\timeInt\to\C^{m,n}$ we understand a factorization of the form
\begin{equation}\label{eq:NSD_generic}
B(t) = \ct{Z(t)}\bmat{ B_{11}(t) & 0 \\ 0 & 0}V(t),
\end{equation}
where $B_{11}(t)$ is an invertible matrix for all $t\in\timeInt$, and $Z:\timeInt\to\GL[m]$ and $V:\timeInt\to\GL[n]$ are pointwise invertible matrix functions, often pointwise unitary.
When $B$ is pointwise Hermitian, we can constrict ourselves to the case $Z=V$, since from \eqref{eq:NSD_generic} we deduce $BV_2=BZ_2=0$ where $Z=[Z_1,Z_2]$ and $V=[V_1,V_2]$, and therefore $\ct{V}BV$ is also a null space decomposition for $B$.

In this section, we are interested in constructing a null space decomposition for the matrix function $Q$ that induces a quadratic storage function.
In applications, $V$ typically represents a time-varying change of variables, thus it is required to have a certain degree of regularity, which in our case would be $V\in W^{1,1}_\loc(\timeInt,\GL[n])$.
However, the construction of a null space decomposition (see e.g.~\cite{Dol64,KunM24}) in the general case is analytically hard and requires quite strong assumptions: $Q$ would have to be as regular as $V$, and crucially it would have to be of constant rank.
This is in contrast with the increased generality that we achieved in \Cref{sec: regularity}, where $Q\in\AUC_\loc(\timeInt,\posSD[n])$ is not even required to be continuous, and no explicit restriction is taken on its rank.

Nevertheless, we will prove that the restrictive assumptions from \cite{Dol64,KunM24} can be removed in the cases that we are interested in.
We start by showing that this can be done in the case of weakly decreasing Hermitian matrix functions.

\begin{theorem}\label{thm:nullspaceDec}
    Let $Q\in\AUC_\loc(\timeInt,\posSD[n])$ be a weakly monotonically decreasing matrix function.
    Then the following statements hold:
    \begin{enumerate}
        \item\label{it:nullspaceDec:1}
        For every $t_{-1},t_0,t_1\in\timeInt$ with $ t_{-1}<t_0<t_1$ we have
        \begin{equation}\label{eq:nullspaceDec:1}
            \ker\pset[\big]{Q(t_{-1})} \subseteq \ker\pset*{\lim_{t\to t_0^-}Q(t)} \subseteq \ker\pset[\big]{Q(t_0)} \subseteq \ker\pset*{\lim_{t\to t_0^+}Q(t)} \subseteq \ker\pset[\big]{Q(t_{1})}.
        \end{equation}
        \item\label{it:nullspaceDec:2}
        For every $t_{-1},t_0,t_1\in\timeInt$ with $t_{-1}\leq t_0\leq t_1$ we have
        \begin{equation}\label{eq:nullspaceDec:2}
            \rank\pset[\big]{Q(t_{-1})} \geq \rank\pset*{\lim_{t\to t_0^-}Q(t)} \geq \rank\pset[\big]{Q(t_0)} \geq \rank\pset*{\lim_{t\to t_0^+}Q(t)} \geq \rank\pset[\big]{Q(t_{1})}.
        \end{equation}
        In particular, $r\coloneqq\rank\circ\,Q:\timeInt\to\set{0,1,\ldots,n}$ is weakly decreasing.
        \item\label{it:nullspaceDec:3}
        There exists $V\in\GL[n]$ such that $\wt Q\coloneqq\ct{V}QV\in\AUC_\loc(\timeInt,\posSD[n])$ is of the form
        \begin{subequations}\label{eq:nullspaceDec:3-4}
        \begin{equation}\label{eq:nullspaceDec:3}
            \wt Q(t) = \bmat{ \wt Q_{11}(t) & 0 \\ 0 & 0_{n-r(t)} }
        \end{equation}
        and satisfies
        \begin{equation}\label{eq:nullspaceDec:4}
            \lim_{s\to t^\pm}\wt Q(s) = \bmat{\wt Q_{11}^\pm(t) & 0 \\ 0 & 0_{n-r_\pm(t)}},
        \end{equation}
        \end{subequations}
        where $\wt Q_{11}(t)\in\posDef[r(t)]$ and $\wt Q_{11}^\pm(t)\in\posDef[r_\pm(t)]$ with $r_\pm(t)\coloneqq\lim_{s\to t^\pm}r(t)$
        for all $t\in\timeInt$.
    \end{enumerate}
\end{theorem}
\begin{proof}
    \begin{enumerate}
        \item We show the four inclusions separately.
        If $x\in\ker(Q(t_{-1}))$, then
        \[
        0 \leq \ct{x}Q(t)x \leq \ct{x}Q(t_{-1})x = 0 \implies \ct{x}Q(t)x = 0 \implies Q(t)x = 0
        \]
        for all $t\geq t_{-1}$. Thus, $(\lim_{t\to t_0^-}Q(t))x=\lim_{t\to t_0^-}(Q(t)x)=0$, i.e., $x\in\ker(\lim_{t\to t_0^-}Q(t))$.

        If $x\in\ker(\lim_{t\to t_0^-}Q(t))$, then since $\ct{x}Q(t_0)x\leq\ct{x}Q(t)x$ holds for every $t\leq t_0$, we obtain in the limit
        \[
        0 \leq \ct{x}Q(t_0)x \leq \lim_{t\to t_0^-}\pset[\big]{\ct{x}Q(t)x} = \ct{x}\pset*{\lim_{t\to t_0^-}Q(t)}x = 0 \implies Q(t_0)x=0,
        \]
        i.e., $x\in\ker(Q(t_0))$.

        For $x\in\ker(Q(t_0))$ we have that for all $t\geq t_0$,
        \[
        0 \leq \ct{x}Q(t)x \leq \ct{x}Q(t_0)x = 0 \implies Q(t)x =0,
        \]
         and therefore $(\lim_{t\to t_0^+}Q(t))x=\lim_{t\to t_0^+}(Q(t)x)=0$, i.e., $x\in\ker(\lim_{t\to t_0^+}Q(t))$.

        Finally, let $x\in\ker(\lim_{t\to t_0^+}Q(t))$. For every $t\leq t_1$ we have $\ct{x}Q(t_1)x\leq\ct{x}Q(t)x$, and thus
        \[
        0 \leq \ct{x}Q(t_1)x \leq \lim_{t\to t_0^+}\pset[\big]{\ct{x}Q(t)x} = \ct{x}\pset*{\lim_{t\to t_0^+}Q(t)}x =0 \implies Q(t_1)x = 0,
        \]
        i.e., $x\in\ker(Q(t_1))$.
        \item Since $\rank(Q(t))=n-\dim(\ker(Q(t)))$ for all $t\in\timeInt$, we deduce immediately \eqref{eq:nullspaceDec:2} from \eqref{eq:nullspaceDec:1}.
        In particular, we have $\rank(Q(t_0))\geq\rank(Q(t_1))$ whenever $t_0,t_1\in\timeInt,\ t_0\leq t_1$, i.e., $r=\rank\circ\,Q$ is a weakly decreasing map.
        \item Because of \ref{it:nullspaceDec:2}., the rank of $Q$ drops only at $K\in\N_0$ points $t_1,\ldots,t_K\in\timeInt,\ t_1<\ldots<t_K$.
        Let us consider any two points $t_0,t_{K+1}\in\timeInt$ with $t_0<t_1$ and $t_{K+1}>t_K$.
       Then the chain of inclusions
        \[
        \begin{split}
            \ker\pset[\big]{Q(t_{0})} &\subseteq \ker\pset*{\lim_{t\to t_1^-}Q(t)} \subseteq \ker\pset[\big]{Q(t_1)} \subseteq \ker\pset*{\lim_{t\to t_1^+}Q(t)} \\
            &\subseteq \ker\pset*{\lim_{t\to t_2^-}Q(t)} \subseteq \ldots \subseteq \ker\pset*{\lim_{t\to t_K^+}Q(t)} \subseteq \ker\pset[\big]{Q(t_{K+1})} \subseteq \C^n
        \end{split}
        \]
        holds, because of \eqref{eq:nullspaceDec:1}.
        
        We then construct iteratively a sequence $\mathcal B_0\subseteq\mathcal B_1^-\subseteq\mathcal B_1\subseteq\mathcal B_1^+\subseteq\mathcal B_2^-\subseteq\ldots\subseteq\mathcal B_K^+\subseteq\mathcal B_{K+1}\subseteq\mathcal B$ of sets of linearly independent vectors in $\C^n$, such that $\mathcal B_k$ is a basis of $\ker(Q(t_k))$ for $k=0,\ldots,K+1$, $\mathcal B_k^\pm$ is a basis of $\ker(\lim_{t\to t_k^\pm}Q(t))$ for $k=1,\ldots,K$, and $\mathcal B$ is a basis of $\C^n$.
        In particular, let $\mathcal B=\set{v_1,\ldots,v_n}$ be such that $\mathcal B_k=\set{v_1,\ldots,v_{j_k}}$ for $k=0,\ldots,K+1$ and $\mathcal B_k^\pm=\set{v_1,\ldots,v_{j_k^\pm}}$ for $k=1,\ldots,K$, for appropriate $j_k$ and $j_k^\pm$.
        The matrix $V\coloneqq[v_n,\ldots,v_1]\in\GL[n]$ then clearly satisfies the desired properties. \qedhere
    \end{enumerate}
\end{proof}

\noindent Since the proof of \Cref{thm:nullspaceDec}\ref{it:nullspaceDec:3} is constructive, we make the construction explicit by reformulating it as the following corollary.

\begin{corollary}\label{cor:NSD_dec}
    Let $Q\in\AUC_\loc(\timeInt,\posSD[n])$ be weakly decreasing. Then there exist times $t_0<t_1<\ldots<t_K$, linearly independent vectors $v_1,\ldots,v_n\in\C^n$, and indices 
    $j_0\leq j_1^-\leq j_1\leq j_1^+\leq\ldots\leq j_K^-\leq j_K\leq j_K^+$,
    such that:
    
    \begin{enumerate}
        \item $\ker(Q(t_k))$ is spanned by $v_1,\ldots,v_{j_k}$ for $k=1,\ldots,K$,
        \item $\ker(\lim_{s\to t_k^-}Q(s))$ is spanned by $v_1,\ldots,v_{j_k^-}$ for $k=1,\ldots,K$,
        \item $\ker(\lim_{s\to t_k^+}Q(s))$ is spanned by $v_1,\ldots,v_{j_k^+}$ for $k=1,\ldots,K-1$,
        \item $\ker(Q(t))$ is spanned by $v_1,\ldots,v_{j_0}$ for all $t<t_1$,
        \item $\ker(Q(t))=\ker(\lim_{s\to t_k^+}Q(s))$ for all $t\in(t_k,t_{k+1})$ and $k=1,\ldots,K$,
        \item $\ker(Q(t))$ is spanned by $v_1,\ldots,v_{j_K}$ for all $t>t_K$.
    \end{enumerate}
\end{corollary}

\begin{remark}\label{rem:rank}
    It is important to emphasize that the size of $\wt Q_{11}$ in \eqref{eq:nullspaceDec:3} is time-varying, since the dimension of $\ker(Q(t))$ is not constant.
    However, in every subinterval $[t_0,t_1]\subseteq\timeInt$ where $r=\rank(Q(t))$ is constant, since $\wt Q|_{[t_0,t_1]}\in\AUC([t_0,t_1],\posSD[n])$, we have that  $\wt Q_{11}|_{[t_0,t_1]}\in\AUC([t_0,t_1],\posSD[r])$.
    The same observations apply to $\wt Q_{11}^\pm$ in \eqref{eq:nullspaceDec:4}.
\end{remark}
\noindent We can now extend the null space decomposition \eqref{eq:nullspaceDec:3-4} to general $Q\in\AUC_\loc(\timeInt,\posSD[n])$ defining quadratic storage functions.
\begin{theorem}\label{thm:NSD_Q}Suppose that $\quadSt{Q}$ is a storage function for a passive LTV system \eqref{eq:tv_system} that is induced by $Q\in\AUC_\loc(\timeInt,\posSD[n])$.
    Then the following statements hold:
    \begin{enumerate}
        \item $r\coloneqq\rank\circ\,Q:\timeInt\to\set{0,1,\ldots,n}$ is weakly decreasing.
        \item There exists $V\in W^{1,1}_\loc(\timeInt,\GL[n])$ such that $\wt Q\coloneqq\ct{V}QV$ admits a null space decomposition of the form \eqref{eq:nullspaceDec:3-4} and is weakly decreasing.
        \item There exists $U\in W^{1,1}_\loc(\timeInt,\GL[n])$ such that $U$ is pointwise unitary, $\wh Q\coloneqq\ct{U}QU$ admits a null space decomposition of the form \eqref{eq:nullspaceDec:3-4}.
        \item If $Q\in W^{1,1}_\loc(\timeInt,\posSD[n])$, then $\wt Q\in W^{1,1}_\loc(\timeInt,\posSD[n])$ and $\wh Q\in W^{1,1}_\loc(\timeInt,\posSD[n])$ also hold.
    \end{enumerate}
\end{theorem}

\begin{proof}
    \begin{enumerate}
        \item Let $t_0\in\timeInt$, let $X\in W^{1,1}_\loc(\timeInt,\GL[n]$ be the fundamental solution matrix associated to $\dot x=Ax$ and $t_0$, and let $Q^X\in\AUC_\loc(\timeInt,\posSD[n])$ be defined as in \Cref{lem:passiveInequality}.
        Since $Q^X$ is weakly decreasing, it satisfies the properties of \Cref{thm:nullspaceDec}, in particular $\rank\circ\,Q^X$ is weakly decreasing, and there exists $V_0\in\GL[n]$ such that $\ct{V_0}Q^XV_0\in\AUC_\loc(\timeInt,\posSD[n])$ has the form \eqref{eq:nullspaceDec:3} and satisfies \eqref{eq:nullspaceDec:4}.
        Note that
        \[
        \rank\circ\,Q^X(t) = \rank\pset[\big]{ \ct{X(t)}Q(t)X(t) } = \rank\pset[\big]{Q(t)} = r(t)
        \]
        for all $t\in\timeInt$, since $X(t)\in\GL[n]$, thus $r=\rank\circ\,Q$ is also weakly decreasing.
        \item It holds that
        \[
        \bmat{ \wt Q_{11}(t) & 0 \\ 0 & 0 } = \ct{V_0}Q^X(t)V_0 = \ct{V_0}\ct{X(t)}Q(t)X(t)V_0 = \ct{V(t)}Q(t)V(t)
        \]
        for all $t\in\timeInt$, with $V=XV_0\in W^{1,1}_\loc(\timeInt,\GL[n])$, thus $\wt Q\coloneqq\ct{V}QV$ is of the form \eqref{eq:nullspaceDec:3}, and analogously
        \[
        \bmat{\wt Q_{11}^\pm(t) & 0 \\ 0 & 0} = \lim_{s\to t^\pm}\pset[\big]{\ct{V_0}Q^X(s)V_0} = \lim_{s\to t^\pm}\pset[\big]{\ct{V_0}\ct{X(s)}Q(s)X(s)V_0} = \lim_{s\to t^\pm}\wt Q(s)
        \]
        for all $t\in\timeInt$, thus $\wt Q$ satisfies \eqref{eq:nullspaceDec:4}.
        Note that $\wt Q$ inherits from $Q^X$ the property of being weakly decreasing, since
        \[
        \wt Q(t_2) - \wt Q(t_1) = \ct{V_0}Q^X(t_2)V_0 - \ct{V_0}Q^X(t_1)V_0 = \ct{V_0}\pset[\big]{Q^X(t_2)-Q^X(t_1)}V_0 \leq 0
        \]
        for all $t_1,t_2\in\timeInt,\ t_1\leq t_2$.
        \item Applying a pointwise Gram-Schmidt procedure to the columns of $V$ one obtains  the time-varying QL factorization $V=UL$ with positive entries on the diagonal of $L$. In particular, $U,L\in W^{1,1}_\loc(\timeInt,\GL[n])$ with $L$ pointwise lower triangular, and $U$ pointwise unitary see e.g.~\cite[Proposition 2.3]{DieE99}, where the proof works in the same way for functions in $W^{1,1}_\loc(\timeInt,\GL[n])$. Then $L^{-1}\in W^{1,1}_\loc(\timeInt,\GL[n])$ is also pointwise lower triangular (see also \Cref{lem:inverseContinuous}). Introducing $\wh Q\coloneq\ct{U}QU\in\AUC_\loc(\timeInt,\posSD[n])$ and partitioning
        \[
        L(t)^{-1} = \bmat{ \wt L_{11}(t) & 0 \\ \wt L_{12}(t) & \wt L_{22}(t) }, \qquad \wt L_{11}(t) \in \GL[r(t)]
        \]
        for all $t\in\timeInt$, we obtain
        \[
            \wh Q(t) = \ct{U(t)}Q(t)U(t) = \ict{L(t)}\wt Q(t)L(t)^{-1} = \bmat{\ct{\wt L_{11}(t)}\wt Q_{11}(t)\wt L_{11}(t) & 0 \\ 0 & 0} = \bmat{\wh Q_{11}(t) & 0 \\ 0 & 0}
        \]
        with $\wh Q_{11}(t)\coloneqq\ct{\wt L_{11}(t)}\wt Q_{11}(t)\wt L_{11}(t)\in\posDef[r(t)]$, which has the desired structure for \eqref{eq:nullspaceDec:3}.
        The same holds for \eqref{eq:nullspaceDec:4}, since with the analogous splitting
        \[
        L(t)^{-1} = \bmat{ \wt L_{11}^\pm(t) & 0 \\ \wt L_{21}^\pm(t) & \wt L_{22}^\pm(t) }, \qquad \wt L_{11}^\pm(t) \in \GL[r_\pm(t)]
        \]
        for all $t\in\timeInt$, we obtain
        \begin{align*}
            \lim_{s\to t^\pm}\wh Q(s) &= \lim_{s\to t^\pm}\ct{U(s)}Q(s)U(s) = \ict{L(t)}\pset*{\lim_{s\to t^\pm}\wt Q(s)}L(t)^{-1} \\
            &= \bmat{\wt L_{11}^\pm(t)\wt Q_{11}^\pm(t)\wt L_{11}^\pm(t) & 0 \\ 0 & 0} = \bmat{\wh Q_{11}^\pm(t) & 0 \\ 0 & 0},
        \end{align*}
        with $\wh Q_{11}^\pm(t)\coloneqq\wt L_{11}^\pm(t)\wt Q_{11}^\pm(t)\wt L_{11}^\pm(t)\in\posDef[r_\pm(t)]$.
        \item This statement is obvious, since absolutely continuous functions are closed under both addition and multiplication.\qedhere
    \end{enumerate}
\end{proof}

\begin{remark}
    We stress that, since the proofs of \Cref{thm:nullspaceDec} and \Cref{thm:NSD_Q} are constructive, these null space decompositions can in principle be numerically computed (see also \Cref{cor:NSD_dec}). In the case of a general weakly decreasing $Q:\timeInt\to\posSD[n]$, it might be difficult to determine the points $t_1,\ldots,t_K$ at which $Q$ decreases its rank without prior knowledge, although they could still be approximated using rank evaluations.
    In the case of the $Q\in\AUC_\loc(\timeInt,\posSD[n])$ inducing a quadratic storage function, the computation of the fundamental solution matrix $X\in W^{1,1}_\loc(\timeInt,\GL[n])$ is also necessary.
\end{remark}

\noindent We illustrate the subtle properties of \Cref{thm:NSD_Q} with an example.
\begin{example}
    Consider the one-dimensional LTV system
	\begin{equation}\label{eq:example_unitaryTransformationsNotSufficient}
	    \dot x(t) =
	    \begin{cases}
	        0 & \text{for }t\leq 0, \\
	        -\frac{2t}{1+t^2}x & \text{otherwise},
	    \end{cases}
	\end{equation}
	in the time interval $\timeInt=\R$.
    Note that the solutions of \eqref{eq:example_unitaryTransformationsNotSufficient} are of the form $x(t)=X(t)x_0$ with $x_0\in\C$ and
    \[
    X(t) =
    \begin{cases}
        1 & \text{for }t\leq 0, \\
        \frac{1}{1+t^2} & \text{for }t>0.
    \end{cases}
    \]
	Define $Q:\R\to\R,\ t\mapsto 1+t^2$, which is positive and continuously differentiable, and therefore $Q\in W^{1,1}_\loc(\R,\posDef[1])$ with $\dot Q(t)=2t$ for all $t\in\R$.
    The storage function defined by $Q$ has the form
    \[
    \quadSt{Q}\pset[\big]{t,x(t)} = \frac{1}{2}(1+t^2)x(t)^2 =
    \begin{cases}
        \frac{(1+t^2)}{2}x_0^2 & \text{for }t\leq 0, \\
        \frac{1}{2(1+t^2)}x_0^2 & \text{for }t>0
    \end{cases}
    \]
    along every solution $x$ of \eqref{eq:example_unitaryTransformationsNotSufficient}.
    Since the system \eqref{eq:example_unitaryTransformationsNotSufficient} has no input, and $\quadSt{Q}$ is decreasing along every solution, we conclude that $\quadSt{Q}$ is in fact a storage function and the system \eqref{eq:example_unitaryTransformationsNotSufficient} is passive.

    Because of \Cref{thm:NSD_Q}, we expect to be able to find $V\in W^{1,1}_\loc(\R,\GL[1])$ such that $\wt Q\coloneqq\ct{V}QV$ is weakly decreasing.
    In fact, by choosing $V=X$ we obtain
    \[
    \wt Q(t) = \ct{V(t)}Q(t)V(t) =
    \begin{cases}
        1+t^2 & \text{for }t\leq 0, \\
        \frac{1}{1+t^2} & \text{for }t<0,
    \end{cases}
    \]
	and thus $\wt Q$ is decreasing.
	However, we cannot find a pointwise unitary $U\in W^{1,1}_\loc(\R,\GL[1]))$ that satisfies the same property. If there were such a $U$, then $\wh Q\coloneqq\ct{U}QU\in W^{1,1}_\loc(\R,\posDef[1])$ would be weakly decreasing, but it would also satisfy $\norm{\wh Q(t)}=\norm{\ct{U(t)}Q(t)U(t)}=\norm{Q(t)}=1+t^2$ for all $t\in\R$.
	Since $\wh Q(t)>0$ for all $t\in\R$, we necessarily have $Q=\wh Q$, and therefore $\wh Q$ cannot be weakly decreasing.
\end{example}

\noindent On the one hand, we observe that \Cref{thm:nullspaceDec} applies to all weakly decreasing pointwise Hermitian positive semidefinite matrix functions, regardless of their possible connection to a control system. Moreover, while \Cref{thm:NSD_Q} requires the involvement of an LTV system, a more careful analysis of its proof reveals that we only need the associated homogeneous differential equation $\dot x(t)=A(t)x(t)$ (or equivalently the closed system obtained by setting $u\equiv 0$) to be passive with storage function $\quadSt{Q}$.

On the other hand, we can actually show that a pointwise Hermitian positive semidefinite matrix function only admits such a null space decomposition if there exists a passive LTV system for which it is a storage function.

\begin{theorem}
    Let $Q:\timeInt\to\posSD[n]$. Then there exists an LTV system \eqref{eq:tv_system} for which $\quadSt{Q}$ as in \eqref{eq:quadStorFunc} is a storage function if and only if there exists $V\in W^{1,1}_\loc(\timeInt,\GL[n])$ such that $\wt Q\coloneqq\ct{V}QV$ is weakly decreasing.
\end{theorem}
\begin{proof}
    Because of \Cref{thm:NSD_Q}, it is clear that the existence of such a $V$ is a necessary condition.
    Suppose now that such $V$ exists, and consider the stationary system $\dot{\wt x}(t)=0$, for which $\wt Q$ as in \eqref{eq:nullspaceDec:3} trivially defines a storage function $\quadSt{\wt Q}$.
    The change of variables $x(t)=V(t)\wt x(t)$ then leads us to another passive system, which has the form $\dot x(t)=\dot V(t)V(t)^{-1}x(t)$ and storage function $\quadSt{\ict{V}\wt QV}=\quadSt{Q}$.
\end{proof}

\section[A necessary condition on the kernel of Q]{A necessary condition on the kernel of $Q$}\label{sec:kernel result} 

We aim to provide a necessary condition for the kernel of the $Q$ matrix function defining quadratic storage functions of passive LTV systems, generalizing \cite[Proposition 11]{CheGH23}. For this, we start by proving the following two lemmas.
\begin{lemma}\label{lem:positiveIntegrals}
    Let $f\in L^1_\loc(\timeInt,\R)$ be such that $\int_{t_0}^{t_1}f(t)\td t\geq 0$ holds for all $t_0,t_1\in\timeInt,\ t_0\leq t_1$. Then $f(t)\geq 0$ for a.e.~$t\in\timeInt$.
\end{lemma}

\begin{proof}
    Let $\timeInt_N\subseteq\timeInt$ for $N\in\N$ be any sequence of compact intervals such that $\timeInt_N\subseteq\timeInt_{N+1}$ for all $N\in\N$ and $\timeInt=\bigcup_{N\in\N}\timeInt_N$, let $E\coloneqq\set{t\in\timeInt\mid f(t)<0}$ and let $E_N\coloneqq E\cap\timeInt_N$ for all $N\in\N$.

    Let us fix any $N\in\N$.
    Since $E_N$ is measurable, for every $n\in\N$ there exists an open set $\Omega_n\subseteq\timeInt$ such that $E_N\subseteq\Omega_n$ and $\abs{\Omega_n\setminus E_N}<\frac{1}{n}$. We assume without loss of generality that $\Omega_{n+1}\subseteq\Omega_n$ for all $n\in\N$, up to replacing $\Omega_{n+1}$ with $\Omega_n\cap\Omega_{n+1}$.
    For every $n\in\N$ the set $\Omega_n$ is open, and therefore there are at most countably many disjoint intervals $I_{n,k}$ such that $\Omega_n=\bigcup_k I_{n,k}$. In particular, we have that
    \[
    \int_{\Omega_n}f(t)\td t = \sum_k \int_{I_{n,k}}f(t)\td t \geq 0
    \]
    for all $n\in\N$.
    It follows that
    \[
    \int_{E_N} f(t)\td t = \int_{\Omega_n} f(t)\td t - \int_{\Omega_n\setminus E_N}f(t)\td t \geq - \int_{\Omega_n\setminus E_N}f(t)\td t,
    \]
    and, therefore,
    \[
    \int_{E_N} f(t)\td t \geq -\lim_{n\to\infty}\int_{\Omega_n\setminus E_N}f(t)\td t = 0,
    \]
    because of the dominated convergence theorem. Since $f<0$ on $E_N$, we conclude that necessarily $\abs{E_N}=0$.

    Thus, since $E=\bigcup_{N\in\N}E_N$, we deduce that
    \[
    \abs{E} = \abs*{\bigcup_{N\in\N}E_N} \leq \sum_{N\in\N}\abs{E_N} = 0,
    \]
    thus $f\geq 0$ a.e.~on $\timeInt$.
\end{proof}

\begin{lemma}\label{lem:nullIntegrals}
    Let $B\in L^1_\loc(\timeInt,\C^{n,m})$ be such that $\int_{t_0}^{t_1}B(t)\td t=0$ for all $t_0,t_1\in\timeInt,\ t_0\leq t_1$. Then $B(t)=0$ for a.e.~$t\in\timeInt$.
\end{lemma}
\begin{proof}
    If $B=[b_{ij}]\in L^1_\loc(\timeInt,\R^{2n,2m})$, then $b_{ij}\in L^1_\loc(\timeInt,\R)$ satisfies $\int_{t_0}^{t_1}b_{ij}(t)\td t=0$ for all $t_0,t_1\in\timeInt,\ t_0\leq t_1$ and all $i,j$. In particular, by applying \Cref{lem:positiveIntegrals} to both $b_{ij}$ and $-b_{ij}$, we obtain $b_{ij}(t)=0$ for a.e.~$t\in\timeInt$, for all $i,j$.
    We conclude that also $B(t)=0$ for a.e.~$t\in\timeInt$.
\end{proof}

\noindent We can now prove a necessary condition for the kernel of the induced matrix function of a quadratic storage function of a passive LTV system.

\begin{theorem}\label{thm:kernelInclusion}
    Suppose that $\quadSt{Q}$ as in \eqref{eq:quadStorFunc} with $Q\in\AUC_\loc(\timeInt,\posSD[n])$ is a storage function for a passive LTV system \eqref{eq:tv_system}.
    Then
    \begin{equation}\label{eq:kernelInclusion}
        \ker\pset[\big]{Q(t)}\subseteq\ker\pset[\big]{Q(t)A(t)+\dot Q(t)}\cap\ker\pset[\big]{C(t)}
    \end{equation}
    holds for a.e.~$t\in\timeInt$.
\end{theorem}

\begin{proof}
    \Cref{thm:NSD_Q} implies that the rank of $Q(t)$ decreases at only a finite number of time instances.
    In particular, we can write $\timeInt$ as the disjoint union of finitely many open intervals where $Q(t)$ has constant rank, plus the points where the rank drops, which form a set of measure zero.
    It is then sufficient to prove the statement of this lemma under the assumption that $Q$ has constant rank on $\timeInt$.
    
    Let $t_0\in\timeInt$, let $V_0\in\C^{n,r}$ be a constant matrix whose columns form a basis of $\ker(Q(t_0))$, let $X\in W^{1,1}_\loc(\timeInt,\GL[n])$ be the fundamental solution matrix associated to $\dot x=Ax$ and $t_0$, let $V\coloneqq XV_0\in W^{1,1}_\loc(\timeInt,\C^{n,r})$, and let $Q^X\in\AUC_\loc(\timeInt,\posSD[n])$ be defined as in \Cref{lem:passiveInequality}.
    Because of the null space decomposition in \Cref{thm:nullspaceDec} and of the constant rank assumption, 
    $\ker(Q(t_0))=\ker(Q^X(t_0))=\ker(Q^X(t))$ holds for all $t\in\timeInt$. Therefore $Q(t)V(t)=Q(t)X(t)V_0=\ict{X(t)}Q^X(t)V_0=0$ for all $t\in\timeInt$.
    On the one hand, the columns of $V(t)$ form a basis of $\ker(Q(t))$ for all $t\in\timeInt$.
    On the other hand, since the classical time derivative of $Q$ exists a.e.~on $\timeInt$, we deduce that
    \[
    \pset[\big]{Q(t)A(t)+\dot Q(t)}V(t) = Q(t)\dot{V}(t) + \dot Q(t)V(t) = \dd{}{t}(QV)(t) = 0
    \]
    for a.e.~$t\in\timeInt$.
    Thus it follows that $\ker(Q(t))\subseteq\ker(Q(t)A(t)+\dot Q(t))$ for a.e.~$t\in\timeInt$.
    
    Suppose now for the sake of contradiction that it is not true that $\ker(Q(t))\subseteq\ker(C(t))$, i.e., that it is not true that $C(t)V(t)=0$ for a.e.~$t\in\timeInt$.
    Due to \Cref{lem:nullIntegrals}  there exist $t_0,t_1\in\timeInt,\ t_0\leq t_1$ such that $\int_{t_0}^{t_1}C(t)V(t)\td t\neq 0$, and therefore there exist $u_0\in\C^m$ and $v_0\in\C^r$ such that
    \[
    \alpha \coloneqq \int_{t_0}^{t_1}\ct{u_0}C(t)V(t)v_0\td t = \ct{u_0}\pset*{\int_{t_0}^{t_1}C(t)V(t)\td t}v_0 \in \C\setminus\set{0}.
    \]
    Let us now consider for every $\lambda\in\R$ the state solution $x_\lambda\in W^{1,1}_\loc(\timeInt,\C^n)$ of \eqref{eq:tv_system}, uniquely determined by the initial condition $x_\lambda(t_0)=\lambda\cc{\alpha}V_0v_0$ and the constant input $u\equiv u_0$, i.e.,
    \[
    x_\lambda(t) = \lambda\cc{\alpha}V(t)v_0 + b_u(t), \qquad b_u(t) \coloneqq X(t)\int_{t_0}^t X(s)^{-1}B(s)u(s)\td s
    \]
    for all $t\in\timeInt$, where notably $b_u\in W^{1,1}_\loc(\timeInt,\C^n)$ does not depend on $\lambda$, and $b_u(0)=0$ (see also \Cref{thm:fundamentalSolution}).
    On the one hand, the difference of the storage function between $t_0$ and $t_1$ is
    \begin{align*}
        \quadSt{Q}\pset[\big]{t_1,x_\lambda(t_1)} - \quadSt{Q}\pset[\big]{t_0,x_\lambda(t_0)}
        &= \frac{1}{2}\ct{\pset[\big]{\lambda\cc{\alpha}V(t_1)v_0 + b_u(t_1)}}Q(t_1)\pset[\big]{\lambda\cc{\alpha}V(t_1)v_0 + b_u(t_1)} \\
        &- \frac{1}{2}\lambda^2\abs{\alpha}^2\ct{v_0}\ct{V_0}Q(t_0)V_0v_0
        = \frac{1}{2}\ct{b_u(t_1)}Q(t_1)b_u(t_1),
    \end{align*}
    in particular, it does not depend on $\lambda$.
    On the other hand, the supply function is
    \begin{align*}
        \int_{t_0}^{t_1}\realPart\pset[\big]{\ct{y(t)}u(t)}\td t
        &= \int_{t_0}^{t_1}\realPart\pset[\big]{\ct{u(t)}C(t)x_\lambda(t)}\td t + \int_{t_0}^{t_1}\realPart\pset[\big]{\ct{u(t)}D(t)u(t)}\td t \\
        &= \realPart\pset*{\cc{\alpha}\lambda\int_{t_0}^{t_1}\ct{u_0}C(t)V(t)v_0\td t} + \int_{t_0}^{t_1}\realPart\pset*{ \ct{u(t)}C(t)b_u(t) + \ct{u(t)}D(t)u(t) }\td t \\
        &= \realPart(\lambda\alpha\cc{\alpha}) + \int_{t_0}^{t_1}c_u(t)\td t = \lambda\abs{\alpha}^2 + \int_{t_0}^{t_1}c_u(t)\td t,
    \end{align*}
    where $c_u\in L^1_\loc(\timeInt,\R)$ does not depend on $\lambda$.
    Therefore, the dissipation inequality
    \[
    \frac{1}{2}\ct{b_u(t_1)}Q(t_1)b_u(t_1) \leq \lambda\abs{\alpha}^2 + \int_{t_0}^{t_1}c_u(t)\td t,
    \]
    must hold for all $\lambda\in\R$. However, this inequality is false for arbitrarily large and negative $\lambda$ which is a contradiction and we conclude that $\ker(Q(t))\subseteq\ker(C(t))$ holds for a.e.~$t\in\timeInt$.
\end{proof}

\noindent We conclude this section by presenting an example where the properties that we identified in this paper are relevant.

\begin{example}
	Let us consider a mass-spring-damper system with time-varying stiffness on the time interval $\timeInt=\R$:
	\begin{equation}\label{eq:msd_second}
		m\ddot z(t) = -k(t)z(t) - b\dot z(t) + u(t),
	\end{equation}
	where $z\in\R$ is the displacement of the mass, $k\in L^1_\loc(\R)$ with $k\geq0$ a.e.~in $\R$ is the time-varying spring stiffness, $b>0$ is the damping factor, and the input $u\in L^2_\loc(\R)$ represents any external force.
	By using the displacement $z\in\R$ and the momentum $p=m\dot z$ as variables, we can rewrite \eqref{eq:msd_second} as the first-order ordinary differential equation
	\[
		\begin{split}
			\dot z(t) &= \frac{1}{m}p(t), \\
			\dot p(t) &= -k(t)z(t) - \frac{b}{m}p(t) + u(t).
		\end{split}
	\]
	By adding the measurement of the velocity $y=\dot z=\frac{p}{m}$ as our output, we obtain then the LTV system
	\begin{equation}\label{eq:msd_ltv}
		\begin{split}
			\bmat{\dot z(t) \\ \dot p(t)} &= \bmat{0 & \frac{1}{m} \\ -k(t) & -\frac{b}{m}}\bmat{z(t) \\ p(t)} + \bmat{0 \\ 1}u(t), \\
			y &= \bmat{0 & \frac{1}{m}}\bmat{z(t) \\ p(t)},
		\end{split}
	\end{equation}
	which has a unique solution $x=(z,p)\in W^{1,1}_\loc(\R,\R^2)$ for every initial condition $x(0)=x_0\in\R^2$ and input $u\in L^2_\loc(\R)$.
	
	A natural choice of storage function for \eqref{eq:msd_ltv} would be its total energy, given by the sum of potential and kinetic energy, i.e.,
	\[
	V(t,z,p) = \frac{1}{2}k(t)z^2 + \frac{1}{2m}p^2 = \frac{1}{2}\bmat{z \\ p}^\top\bmat{k(t) & 0 \\ 0 & \frac{1}{m}}\bmat{z \\ p}.
	\]
	Equivalently, by defining the matrix function $Q=\operatorname{diag}(k,m):\R\to\posSD[2]$, we can write $V=\quadSt{Q}$.
	
	Suppose for now that $k\in W^{1,1}_\loc(\R,\R)$. We observe that, for every state-input-output solution $(x,u,y)$ of \eqref{eq:msd_ltv} and every $t_0,t_1\in\R,\ t_0\leq t_1$, it holds that
	\begin{align*}
	&\quadSt{Q}\pset[\big]{t_1,x(t_1)} - \quadSt{Q}\pset[\big]{t_0,x(t_0)}
	= \int_{t_0}^{t_1}\dd{}{t}\quadSt{Q}\pset[\big]{t,x(t)}\td t
	= \int_{t_0}^{t_1}\pset*{ \frac{1}{2}\ct{x}\dot Qx + \ct{x}Q\dot x}\td t \\
	&\qquad= \int_{t_0}^{t_1}\pset*{ \frac{1}{2}\dot kz^2 + kz\dot z + \frac{1}{m}p\dot p }\td t
	= \int_{t_0}^{t_1}\pset*{ \frac{1}{2}\dot kz^2 - b\frac{p^2}{m^2}}\td t + \int_{t_0}^{t_1}yu\td t,
	\end{align*}
	Even under these additional regularity assumptions, it is evident that $\quadSt{Q}$ is a storage function if and only if $\dot k\leq 0$ a.e.~in $\R$. In fact, because of \Cref{lem:positiveIntegrals}, $\quadSt{Q}$ is a storage function if and only if $\frac{1}{2}\dot kz^2\leq b\frac{p^2}{m^2}$ a.e.~in $\R$ along every solution.
	Then, since the constant state trajectory $x\equiv[1,0]^\top$ is a solution for $u=k$, we obtain $\frac{1}{2}\dot k\leq 0$ a.e.~in $\R$.
	
	Let us now assume that $k$ is decreasing due to wear and tear, until a certain time point where the spring breaks and $k$ is set to zero. More precisely, consider the stiffness defined as
	\[
	k(t) =
	\begin{cases}
		1 & \text{for }t<0, \\
		1 - \frac{2}{3}t & \text{for }0\leq t<1, \\
		0 & \text{for }t\geq 1,
	\end{cases}
	\]
	which clearly satisfies $k\in L^1_\loc(\R)$ and $k(t)\geq 0$ for all $t\in\R$.
	In fact, we also have $k\in\AUC_\loc(\R,\R)$, since it can be split into $k=k_a+k_s$ with
	\[
	k_a(t) =
	\begin{cases}
		0 & \text{for }t<0, \\
		-\frac{2}{3}t & \text{for }0\leq t<1, \\
		-\frac{2}{3} & \text{for }t\geq 1,
	\end{cases}\qquad
	k_s(t) = 
	\begin{cases}
		1 & \text{for }t<1, \\
		\frac{2}{3} & \text{for }t\geq 1,
	\end{cases}\qquad
	\dot k(t) =
	\begin{cases}
		0 & \text{for }t<0, \\
		-\frac{2}{3} & \text{for }0<t<1, \\
		0 & \text{for }t>0,
	\end{cases}
	\]
	in particular $k_a\in W^{1,1}_\loc(\R,\R)$, and $k_s\in\BV_\loc(\R,\R)$ is the weakly decreasing singular part of $k$.
	Since $m>0$ is constant, we deduce immediately that $Q\in\AUC_\loc(\R,\posSD[2])$, thus $\quadSt{Q}$ is a suitable storage function candidate.
    Note that $Q\notin W^{1,1}_\loc(\R,\posSD[2])$, since $Q$ is not even continuous.
    Furthermore, $Q$ admits a null space decomposition that can be obtained simply by swapping the two variables, i.e.,
	\[
	U = \bmat{0 & 1 \\ 1 & 0}, \text{ \quad which implies \quad} U^\top Q(t)U = \bmat{\frac{1}{m} & 0 \\ 0 & k(t)},
	\]
	since $k(t)>0$ for $t<1$ and $k(t)=0$ for $t\geq 1$. In particular, note that $Q(t)$ has rank two for $t<1$, rank one for $t\geq 1$.

    Moreover, we verify that the necessary kernel condition from \Cref{thm:kernelInclusion} is satisfied.
    In fact, from the definition of $k$ and the structure of the equation coefficients we have
    \begin{align*}
        &\ker\pset[\big]{Q(t)} =
        \begin{cases}
            \{0\} & \text{for }t<1, \\
            \set{(z,0) \mid z\in\R} & \text{for }t>1,
        \end{cases}
        \\
        &\ker\pset[\big]{Q(t)A(t)+\dot Q(t)} =
        \begin{cases}
            \{0\} & \text{for }t<1, \\
            \set{(z,0) \mid z\in\R} & \text{for }t>1,
        \end{cases}
        \\
        &\ker\pset[\big]{C(t)} =
        \set{(z,0) \mid z\in\R} \quad \text{for all }t\in\R,
    \end{align*}
    which indeed satisfy \eqref{eq:kernelInclusion}.

	Finally, we show that $\quadSt{Q}$ is in fact a storage function for \eqref{eq:msd_ltv} for this choice of $k$.
	Let $t_0,t_1\in\R,\ t_0\leq t_1$; we distinguish three cases:
	\begin{description}
		\item[($t_1<1$)] Since $k|_{(-\infty,1)}\in W^{1,1}_\loc((-\infty,1),\R)$, we deduce like before that
		\[
		\quadSt{Q}\pset[\big]{t_1,x(t_1)} - \quadSt{Q}\pset[\big]{t_0,x(t_0)}
		= \int_{t_0}^{t_1}\pset*{ \frac{1}{2}\dot kz^2 - b\frac{p^2}{m^2}}\td t + \int_{t_0}^{t_1}yu\td t
		\leq \int_{t_0}^{t_1}yu\td t.
		\]
		\item[($t_0\geq 1$)] Since $k|_{(1,\infty)}\in W^{1,1}_\loc((1,\infty),\R)$, we deduce again that
		\[
		\quadSt{Q}\pset[\big]{t_1,x(t_1)} - \quadSt{Q}\pset[\big]{t_0,x(t_0)}
		= -\int_{t_0}^{t_1}b\frac{p^2}{m^2}\td t + \int_{t_0}^{t_1}yu\td t
		\leq \int_{t_0}^{t_1}yu\td t.
		\]
		\item[($t_0<1\leq t_1$)] For every $\varepsilon>0$ smaller than $1-t_0$ we write
		\begin{multline*}
		\quadSt{Q}\pset[\big]{t_1,x(t_1)} - \quadSt{Q}\pset[\big]{t_0,x(t_0)}
		= \quadSt{Q}\pset[\big]{t_1,x(t_1)} - \quadSt{Q}\pset[\big]{1-\varepsilon,x(1-\varepsilon)} \\
		+ \quadSt{Q}\pset[\big]{1-\varepsilon,x(1-\varepsilon)} - \quadSt{Q}\pset[\big]{t_0,x(t_0)}.
		\end{multline*}
		Since $1-\varepsilon<1$, we have
		\[
		\quadSt{Q}\pset[\big]{1-\varepsilon,x(1-\varepsilon)} - \quadSt{Q}\pset[\big]{t_0,x(t_0)} \leq \int_{t_0}^{1-\varepsilon}y(t)u(t)\td t.
		\]
		Since $t_1\geq 1$, we also have
		\begin{align*}
		& \quadSt{Q}\pset[\big]{t_1,x(t_1)} - \quadSt{Q}\pset[\big]{1-\varepsilon,x(1-\varepsilon)}
		= \frac{1}{2m}\pset[\big]{p(t_1)^2-p(1-\varepsilon)^2} - \frac{1}{2}k(1-\varepsilon)z(1-\varepsilon)^2 \\
		&\qquad= \frac{1}{m}\int_{1-\varepsilon}^{t_1}p(t)\dot p(t)\td t - \frac{1}{6}(1+2\varepsilon)z(1-\varepsilon)^2
		\leq \frac{1}{m}\int_{1-\varepsilon}^{t_1}p(t)\pset*{-k(t)z(t)-\frac{b}{m}p(t)+u(t)}\td t \\
		&\qquad\leq -\frac{1}{m}\int_{1-\varepsilon}^1k(t)p(t)z(t)\td t + \int_{1-\varepsilon}^{t_1} y(t)u(t)\td t.
		\end{align*}
		Combining the two terms, we deduce that
		\[
		\quadSt{Q}\pset[\big]{t_1,x(t_1)} - \quadSt{Q}\pset[\big]{t_0,x(t_0)}
		\leq \int_{t_0}^{t_1}y(t)u(t)\td t -\frac{1}{m}\int_{1-\varepsilon}^1k(t)p(t)z(t)\td t
		\]
		holds for all $\varepsilon>0$ sufficiently small.
		Since the second integral converges to zero for $\varepsilon\to 0$, we conclude that once again the dissipation inequality holds.
	\end{description}
	Since these three cases are exhaustive, we conclude that \eqref{eq:msd_ltv} with this $k$ is indeed passive, and the total energy $\quadSt{Q}$ is a storage function for the system.
\end{example}

\section{Conclusion}

In this paper, the properties of quadratic storage functions for linear time-varying systems with very mild regularity assumptions have been studied extensively.
The focus on quadratic storage functions has been justified by showing that the existence of any storage function implies the existence of a quadratic one, for example the available storage of the system.
This simplifies the notably very hard problem of determining a storage function, since the class of functions to consider is greatly simplified.
We have shown that their defining matrix functions $Q$ have to be at least locally absolutely upper semicontinuous, in the extended sense given by the Loewner partial order. While storage functions are often assumed to be quite regular, this is a necessary step to be able to study the passivity of systems with discontinuous coefficients.

A null space decomposition for $Q$ under state transformations has been proven to always exist, unlike in the case of general matrix functions, where stronger regularity assumptions are necessary \cite{Dol64,KunM24}.
This decomposition delivers a better understanding on the structure of $Q$, in particular showing that the kernel of $Q$ cannot reduce dimension in time. It also allows us to extend a necessary condition on the kernel of $Q$, previously known for linear time-invariant port-Hamiltonian systems \cite{CheGH23}, to the case of passive linear time-varying systems.

Many of these properties will be fundamental tools in a followup work, where the differences and connections between dissipativity concepts for linear time-varying systems are analyzed.
Furthermore, the deeper understanding of time-varying storage functions achieved in this paper has the potential of being exploited to develop new techniques to determine Lyapunov functions and novel numerical methods.

\printbibliography

\appendix

\section{Appendix}

\noindent The following matrix generalization of the H\"older inequality is central in our choice of function spaces.

\begin{theorem}[Generalized Hölder inequality]\label{thm:genHolder}
    Let $F\in L^p(\timeInt,\C^{l,m})$ and $G\in L^q(\timeInt,\C^{m,n})$, where $1\leq p,q,r\leq\infty$ and $\frac{1}{p}+\frac{1}{q}=\frac{1}{r}$. Then $FG\in L^r(\timeInt,\C^{l,n})$ with
    $\norm{FG}_{L^r} \leq \norm{F}_{L^p}\norm{G}_{L^q}$.
\end{theorem}
\begin{proof}
    By definition we have
    \[
    \norm{FG}_{L^r} = \norm[\big]{\norm{FG}_2}_{L^r} \leq \norm[\big]{\norm{F}_2\norm{G}_2}_{L^r} \leq \norm[\big]{\norm{F}_2}_{L^p}\norm[\big]{\norm{G}_2}_{L^q} = \norm{F}_{L^p}\norm{G}_{L^q},
    \]
    where we applied the better known scalar generalized Hölder inequality to $\norm{F}_2\in L^p(\timeInt,\R)$ and $\norm{G}_2\in L^q(\timeInt,\R)$, see e.g.~\cite[Chapter 4, Remark 2]{Bre10}
\end{proof}

\noindent In \Cref{subsec:setting} we justified our choice of function spaces mentioning the Carath\'eodory conditions. To be more precise, we provide here the following results.

\begin{theorem}\label{thm:solution_L1}
    Let $A\in L^1_\loc(\timeInt,\C^{n,n})$ and $b\in L^1_\loc(\timeInt,\C^n)$. Then for every $(t_0,x_0)\in\timeInt\times\C^n$ the ordinary differential equation
 \begin{equation}\label{eq:solution_L1}
        \dot x = A(t)x + b(t)
    \end{equation}
    has exactly one solution $x\in W^{1,1}_\loc(\timeInt,\C^n)$ such that $x(t_0)=x_0$.
\end{theorem}

\begin{proof}
    The proof is exactly that of \cite[Theorem 3]{Fil88}, up to identifying $\C$ with $\R^2$.
\end{proof}

\begin{corollary}\label{cor:solution_tv}
    For every initial condition $(t_0,x_0)\in\timeInt\times\C^n$ and input $u\in L^2_\loc(\timeInt,\C^m)$, the LTV system \eqref{eq:tv_system} has exactly one solution $x\in W^{1,1}_\loc(\timeInt,\C^n)$ that satisfies $x(t_0)=x_0$. Furthermore, the corresponding output $y$ is an element of $L^2_\loc(\timeInt,\C^m)$.
\end{corollary}
\begin{proof}
    By H\"older's inequality, we have $b\coloneqq Bu\in L^1_\loc(\timeInt,\C^m)$, and therefore we can apply \Cref{thm:solution_L1} to obtain a unique solution $x\in W^{1,p}_\loc(\timeInt,\C^n)$.
    Applying then the generalized H\"older inequality (\Cref{thm:genHolder}) to the output equation, we get $y=Cx+Du\in L^2_\loc(\timeInt,\C^m)$.
\end{proof}

\begin{lemma}\label{lem:linearitySolutions}
    The state-input-output solutions $(x,u,y)$ of the LTV system \eqref{eq:tv_system} form a linear subspace of $W^{1,1}_\loc(\timeInt,\C^n)\times L^2_\loc(\timeInt,\C^n)\times L^2_\loc(\timeInt,\C^n)$.
\end{lemma}
\begin{proof}
    The statement immediately follows from the fact that \eqref{eq:tv_system} is a homogeneous linear system in the variables $x,\dot x,u,y$.
\end{proof}

\noindent Next we  study the  regularity properties of the fundamental solution matrix associated to the homogeneous differential equation $\dot x=A(t)x$.
For that, we first need the following lemmas.
\begin{lemma}\label{lem:integralIsAC}
    Let $A\in L^1([t_0,t_1],\C^{n,m})$. Then for every $t_0\in\timeInt$ the map
    \[
    \mathbf{A}_{t_0} : \timeInt \to \C, \qquad t\mapsto \int_{t_0}^t A(s)\td s
    \]
    satisfies $\mathbf{A}_{t_0}\in W^{1,1}_\loc(\timeInt,\C^{n,m})$ and $\dot{\mathbf{A}}_{t_0}=A$.
\end{lemma}

\begin{proof}
    Let $A=[a_{ij}]\in L^1_\loc(\timeInt,\R^{2n,2m})$. Then we can equivalently prove the statement for every entry $a_{ij}\in L^1_\loc(\timeInt,\R)$.
    Let $K\subseteq\timeInt$ be any compact subset and let $[a,b]\subseteq\timeInt$ be any compact subinterval such that $K\cup\set{t_0}\subseteq[a,b]$.
    Then, it is sufficient to show that, given $f\in L^1([a,b],\R)$, the function $F:[a,b]\to\R,\ t\mapsto\int_{t_0}^t f(s)\td s$ satisfies $F\in W^{1,1}([a,b],\R)$ and $\dot F=f$.
    Due to \cite[Theorem 20.9 and Lemma 20.14]{Car00}, we see that the map $\wt F:[a,b]\to\R,\ t\mapsto\int_a^tf(s)\td s$ satisfies $\wt F\in W^{1,1}([a,b],\R)$ and $\dot{\wt F}=f$. Since
    \[
    F(t) = \int_a^t f(s)\td s = \int_{t_0}^t f(s)\td s + \int_a^{t_0} f(s)\td s = \wt F(t) + c
    \]
    for all $t\in[a,b]$, with a constant $c\in\R$  depending only on $a,t_0,f$, we deduce that also $F\in W^{1,1}([a,b],\R)$ and $\dot F=\dot{\wt F}=f$.
\end{proof}

\begin{lemma}\label{lem:inverseContinuous}
    Let $X:\timeInt\to\GL[n]$ and $X^{-1}:\timeInt\to\C^{n,n},\ t\mapsto X(t)^{-1}$. Then the following statements hold:
    \begin{enumerate}
        \item\label{it:inverseContinuous:1} If $X\in\mathcal C(\timeInt,\GL[n])$, then  $X^{-1}\in\mathcal C(\timeInt,\GL[n])$.
        \item If $X\in W^{1,1}_\loc(\timeInt,\GL[n])$, then  $X^{-1}\in W^{1,1}_\loc(\timeInt,\GL[n])$.
    \end{enumerate}
\end{lemma}
\begin{proof}
    Using that the inverse is $X^{-1}=\det(X)^{-1}\operatorname{adj}(X)$, where for all $t\in\timeInt$, $\operatorname{adj}(X(t))\in\GL[n]$ denotes the adjugate matrix of $X(t)$.
    \begin{enumerate}
        \item Suppose that $X$ is continuous. Then $\det(X)$ and $\operatorname{adj}(X)$ are also continuous. Since $\det(X(t))\neq 0$ for all $t\in\timeInt$, we have that also $\det(X(t))^{-1}$ is continuous and thus, $X^{-1}=\det(X)^{-1}\operatorname{adj}(X)\in\mathcal C(\timeInt,\GL[n])$.
        \item Due to 
    \ref{it:inverseContinuous:1}., we have that $X^{-1}\in\mathcal C\subseteq L^\infty_\loc\subseteq L^1_\loc$.
        From matrix differential calculus, we know that $\dd{}{t}(X^{-1})=X^{-1}\dot XX^{-1}$ is the weak derivative of $X^{-1}$.
        By the generalized Hölder inequality (\Cref{thm:genHolder}) it follows that $\dd{}{t}(X^{-1})\in L^1_\loc(\timeInt,\GL[n])$, and therefore $X^{-1}\in W^{1,1}_\loc(\timeInt,\GL[n])$.
        \qedhere
    \end{enumerate}
\end{proof}

\noindent We can now study the properties of the fundamental solution matrix.
\begin{theorem}\label{thm:fundamentalSolution}
    Let $A\in L^1_\loc(\timeInt,\C^{n,n})$. Then the following statements hold:
    \begin{enumerate}
        \item For every $t_0\in\timeInt$ the homogeneous ordinary matrix differential equation $\dot X(t)=A(t)X(t)$ has exactly one solution $X\in W^{1,1}_\loc(\timeInt,\C^{n,n})$ such that $X(t_0)=I_n$.
\item\label{it:fundamentalSolution:2} For every $(t_0,x_0)\in\timeInt\times\C^n$ the unique solution of the homogeneous differential equation $\dot x=A(t)x$ that satisfies the initial condition $x(t_0)=x_0$ can be expressed as $x(t)=X(t)x_0$ for all $t\in\timeInt$.
\item $X(t)$ is invertible for all $t\in\timeInt$. In particular $X^{-1}\in W^{1,1}_\loc(\timeInt,\C^{n,n})$.
\item For every $(t_0,x_0)\in\timeInt\times\C^n$ the unique solution $x\in W^{1,1}_\loc(\timeInt,\C^n)$ of the inhomogeneous differential equation \eqref{eq:solution_L1} that satisfies the initial condition $x(t_0)=x_0$, can be expressed as
\begin{equation}\label{eq:solution_with_fundSol}
            x(t) = X(t)\pset*{ x_0 + \int_{t_0}^t X^{-1}(s)b(s)\td s }
        \end{equation}
        for all $t\in\timeInt$.
    \end{enumerate}
\end{theorem}
\begin{proof}
\begin{enumerate}
\item The existence and uniqueness of the solution $X\in W^{1,1}_\loc(\timeInt,\C^{n,n})$ follows immediately from \Cref{thm:solution_L1}, reinterpreting $\dot X=AX$ as $\dd{}{t}\mathrm{vec}(X)=(I_n\otimes A(t))\mathrm{vec}(X)$ under vectorization, where $\otimes$ represents the Kronecker product, and applying \Cref{thm:solution_L1}.
\item Let $x(t)\coloneqq X(t)x_0$. Then $\dot x(t)=\dot X(t)x_0=A(t)X(t)x_0=A(t)x(t)$ and $x(t_0)=A(t_0)x_0=x_0$.
\item Suppose for the sake of contradiction that there exists $t_1\in\timeInt$ for which $X(t_1)$ is singular, i.e. there is $x_1\in\C^n\setminus\{0\}$ such that $X(t_1)x_1=0$.
Because of \ref{it:fundamentalSolution:2}, both $x(t)\coloneqq X(t)x_1$ and $\wt x(t)\equiv 0$ are solutions of $\dot x(t)=A(t)x(t)$ satisfying $x(t_1)=\wt x(t_1)=0$.
It follows that $x=\wt x$, and therefore $x_1=X(t_0)x_1=x(t_0)=\wt x(t_0)=0$, in contradiction with $x_1\neq 0$.
 Therefore,  $X(t)$ is invertible for all $t\in\timeInt$, and by \Cref{lem:inverseContinuous} it follows that $X^{-1}\in W^{1,1}_\loc(\timeInt,\GL[n])$.
\item Define
\[
  f:\timeInt\to\C^{n},\qquad t\mapsto\int_{t_0}^{t}X^{-1}(s)b(s)\td s
\]
and $\wt x\coloneqq x-Xf$.
By the generalized Hölder inequality (\Cref{thm:genHolder}), $X^{-1}b\in L^1_\loc(\timeInt,C^n)$. Thus, $f\in W^{1,1}_\loc(\timeInt,\C^n)$ with $\dot{f}=X^{-1}b$, by \Cref{lem:integralIsAC}.
Since $W^{1,1}_\loc(\timeInt,\C)$ is an algebra, we get from $X\in W^{1,1}_\loc(\timeInt,\C^n)$ that $\wt x=x-Xf\in W^{1,1}_\loc(\timeInt,\C^n)$.
Furthermore, we obtain that
\[
    \dot{\wt x} = \dot x - \dot Xf - X\dot f = Ax + b - AXf - XX^{-1}b = A(x-Xf) = A\wt x
\]
and $\wt x(t_0)=x_0$. Thus $\wt x(t)=X(t)x_0$ for all $t\in\timeInt$, because of \ref{it:fundamentalSolution:2}.
We conclude that
\[
        x(t) = \wt x(t) + X(t)f(t) = X(t)x_0 + X(t)\int_{t_0}^{t}X^{-1}(s)b(s)\td s = X(t)\pset*{ x_0 + \int_{t_0}^{t}X^{-1}(s)b(s)\td s}
\]
holds for all $t\in\timeInt$. \qedhere
    \end{enumerate}
\end{proof}

\end{document}